\documentclass[a4paper,oneside,10pt,final,reqno]{amsart}

\usepackage{amssymb}
\usepackage[arrow,line,matrix]{xy}
\usepackage{version}
\usepackage{color}
\usepackage{showkeys}

\excludeversion{NB}

\newcommand{\bbA} {\mathbb{A}}
\newcommand{\bbC} {\mathbb{C}}
\newcommand{\bbF} {\mathbb{F}}
\newcommand{\bbK} {\mathbb{K}}
\newcommand{\bbQ} {\mathbb{Q}}

\newcommand{\bbZ} {\mathbb{Z}}

\newcommand{\calC} {\mathcal{C}}
\newcommand{\calD} {\mathcal{D}}
\newcommand{\calF} {\mathcal{F}}
\newcommand{\calH} {\mathcal{H}}

\newcommand{\calO} {\mathcal{O}}
\newcommand{\calV} {\mathcal{V}}

\newcommand{\frkh}{\mathfrak{h}}

\newcommand{\ve}{\varepsilon}

\newcommand{\diag}{\mathop{\operatorname{diag}}}
\newcommand{\sign}{\mathop{\operatorname{sign}}}

\newcommand{\Ob}{\operatorname{Ob}}
\newcommand{\SU}{\mathop{\operatorname{SU}}}

\newcommand{\Det}{\mathop{\operatorname{det}} \nolimits}
\newcommand{\Hom}{\mathop{\operatorname{Hom}} \nolimits}
\newcommand{\Ker}{\mathop{\operatorname{Ker}}}
\newcommand{\Lim}{\mathop{\operatorname{lim}} \nolimits}
\newcommand{\Rad}{\mathop{\operatorname{Rad}} \nolimits}

\newcommand{\LHS}{\operatorname{LHS}}
\newcommand{\RHS}{\operatorname{RHS}}

\newcommand{\vir}{\mathcal{V}ir}
\newcommand{\vqt}{\vir_{q,t}}
\newcommand{\vqtp}{\vir_{q,t}^{+}}

\newcommand{\cls}[1]{{\bf Class} $#1$}

\newcommand{\dd}{\oalign{\raisebox{0.4em}[0.5em][0em]{$\circ$} \cr
                         \raisebox{0.15em}[0.5em][0em]{$\circ$}}}
\newcommand{\nord}[1]{\dd #1 \dd}
\newcommand{\pair}[2]{\left< #1 , #2 \right>}

\theoremstyle{plain}
 \newtheorem{thm}{Theorem}[section]
 \newtheorem{lem}[thm]{Lemma}
 \newtheorem{prop}[thm]{Proposition}
 
 \newtheorem*{thm*}{Theorem}
\theoremstyle{definition}
 \newtheorem{dfn}[thm]{Definition}
 
 \newtheorem{fct}[thm]{Fact}
 \newtheorem{rmk}[thm]{Remark}
\theoremstyle{remark}

\numberwithin{equation}{section}
\allowdisplaybreaks

\begin{document}


\title{Norm of the Whittaker vector of the deformed Virasoro algebra}

\author{Shintarou Yanagida}
\address{Research Institute for Mathematical Sciences,
Kyoto University, Kyoto 606-8502, Japan}
\email{yanagida@kurims.kyoto-u.ac.jp}

\date{November 3, 2014}


\begin{abstract}
We give a proof of the recursive formula on the norm of Whittaker vector of 
the deformed Virasoro algebra, which is an analog of the one for 
the Virasoro Lie algebra proposed by Al.~Zamolodchikov. 
Our formula gives a proof of the pure SU(2) 5d AGT relation 
proposed by Awata and Yamada.
We also give a summary of the structures of the deformed Virasoro algebra 
and the fundamental properties of the Verma module.
\end{abstract}

\maketitle
\tableofcontents

\section{Introduction}

\subsection{}
Let $q$ and $t$ be indeterminates and set $\bbF:=\bbQ((q,t)$.
The deformed Virasoro algebra $\vqt$ was introduced in \cite{SKAO} 
as the (topological) algebra over $\bbF$ generated by the current 
$T(z)=\sum_{n\in\bbZ}T_n z^{-n}$ satisfying the relation
$$
 f(w/z) T(z) T(w) -T(w)T(z) f(z/w) = 
 -\dfrac{(1-q)(1-t^{-1})}{1-q/t}
 \Bigl(\delta( q t^{-1} w/z) - \delta(t q^{-1} w/z)\Bigr)
$$
with $f(x) := \exp(\sum_{n\ge1}(1-q^n)(1-t^{-n})/(1+q^n/t^n)\cdot x^n/n)$ 
and $\delta(x) := \sum_{x \in \bbZ} x^n$.

The algebra $\vqt$ has the Verma module, an analogue of the Verma modules of 
triangular decomposed Lie algebras.
It is $\bbZ_{\ge0}$-graded and generated by a highest weight vector.
The Verma module with highest weight $h$ is denoted by $M(h)$,
and its $n$-th grading part is denoted by $M(h)^{(n)}$.
The highest weight vector of $M(h)$ will be denoted by $1_{h}$.
The highest weight condition can be written as 
$T_0.1_h = h 1_{h}$ and $T_n.1_{h} = 0$ ($n \ge 1$).
As in the triangulated Lie algebra,
we have the contravariant form $\pair{\cdot}{\cdot}: M(h) \otimes_{\bbF} M(h) \to \bbF$
on $M(h)$.

In \cite{AY:2010}, Awata and Yamada proposed a conjecture 
which state that the instanton part of Nekrasov's five-dimensional pure $\SU(2)$ 
partition function \cite{N:2003} 
coincides with the norm $\pair{w}{w}$ of the Whittaker vector $w$ 
of the Verma module $M(h)$.
The element $w$ is an analogue of the Whittaker vector \cite{K:1978} for finite dimensional 
Lie algebras.

The Whittaker vector $w$ is an element of the completion 
$\widehat{M}(h)$ of the Verma module $M(h)$ satisfying  
$T_1.w=\xi w$ and $T_n.w =0$ for any $n\in \bbZ_{>1}$.
Here $\xi$ is an indeterminate.

Now the conjecture in  \cite{AY:2010} can be stated as 
\begin{align}\label{eq:qagt}
 \pair{w}{w} \stackrel{?}{=} 
 \sum_{\lambda,\mu} (\xi^2 t/q)^{|\lambda|+|\mu|} Z_{\lambda, \mu}(Q,q,t).
\end{align}
In the right hand side, 
the summation is taken over all the partitions $\lambda$ and $\mu$.
The parameter $Q$ is defined by the relation $h= Q+Q^{-1}$.
We also used the expression
\begin{align*}
&Z_{\lambda,\mu}(Q,q,t) :=
 \dfrac{1}
       {N_{\lambda,\lambda}(1)N_{\mu,\mu}(1)
        N_{\lambda,\mu}(Q)N_{\mu,\lambda}(Q^{-1})},
\\
&N_{\lambda,\mu}(Q) :=
 \prod_{(i,j)\in\mu}    (1-Q q^{\lambda_i-j}t^{\mu'_j-i+1})
 \prod_{(i,j)\in\lambda}(1-Q q^{-\mu_i+j-1}t^{-\lambda'_j+i}),
\end{align*}
where $(i,j) \in \lambda$ means $1\le i \le \ell(\lambda)$ and $1\le j \le \lambda_i$,
and $\lambda'$ is the transpose of the partition $\lambda$.
The right hand side in \eqref{eq:qagt} is the instanton part of 
Nekrasov's five-dimensional pure $\SU(2)$ partition function.

This conjecture is a natural extension of the degenerate version \cite{G:2008} 
of the four-dimensional AGT (Alday-Gaiotto-Tachikawa) relation \cite{AGT:2010}.
The original AGT relation claims that the instanton part of 
Nekrasov's four-dimensional $\SU(2)$ partition function with $N_f=4$ 
anti-fundamental hypermultiplets \cite{N:2003} coincides with 
the conformal block of the Virasoro algebra.

The four-dimensional degenerated version of the AGT relation 
is now extended for a wide class of $W$-algebras \cite[\S6]{BFFR}, 
and is proved for type A \cite{SV:2013}
and for type ADE \cite{BFN}
using geometric representation theory on the instanton moduli spaces.

In this paper, we give a proof of Awata and Yamada's five-dimensional 
conjecture \eqref{eq:qagt}.
Our strategy is to use the Zamolodchikov-type recursive formula.
This strategy was originally proposed by Poghossian \cite{P:2009} 
in the four-dimensional case.
On the conformal field theory side, the recursion formula was
derived in \cite{HJS:2010a} by taking the degeneration limit technique 
of Zamolodchikov's  elliptic recursive formula for the four-point 
conformal block on the sphere \cite{Z:1984}.
On the Nekrasov partition function side, 
Fateev and Litvinov \cite{FL:2010} used 
the integral expression of the four-dimensional Nekrasov function to 
analyze its poles and residues, and they obtained the same recursion formula.

In the five-dimensional case,
the Zamolodchikov-type formula was shown in the Nekrasov side 
by the author's previous work \cite{Y:2010}.
Thus the proof of Awata and Yamada's proposal 
is reduced to the proof or the recursive formula 
for the norm of the Whittaker vector of $\vir_{q,t}$.
Now let us state the main theorem.

\begin{thm*}
Let $\xi$ be an indeterminate  and
$w_\xi \in M(h)$ be the Whittaker vector.
Then we can express $\pair{w_\xi}{w_\xi}$ as
$$
 \pair{w_\xi}{w_\xi} = \sum_{n=0}^\infty (\xi^{2}t/q)^{n} F_n(Q,q,t)
$$ 
with $F_n(Q,q,t) \in \bbQ(Q,q,t)$. 
Moreover $F_n(Q,q,t)$ satisfies the following recursive relation.
\begin{align*}
F_n(Q,q,t)
=\delta_{n,0}+\sum_{\substack{r,s\in\bbZ, \\ 1\le r s\le n}}
 \dfrac{A_{r,s}(q,t)F_{n-r s}(q^r t^{s},q,t)}{Q-q^r t^{-s}}.
\end{align*}
Here $A_{r,s}(q,t)$ is the following rational function.
\begin{align*}
A_{r,s}(q,t) := -\sign(r) q^r t^{-s} 
\prod_{\substack{-|r|\le i\le|r|-1, \\ -|s|\le j\le|s|-1, \\ (i,j)\neq (0,0)}}
\dfrac{1}{1-q^i t^{-j}}
\end{align*}
with $\sign(r)=1$ if $r>0$ and $\sign(r)=-1$ if $r<0$.
\end{thm*}

Let us explain the organization of the paper here.
In \S \ref{sec:basics}  we introduce the deformed Virasoro algebra.
\S \ref{sec:verma} introduces the Verma modules.
In \S \ref{sec:sing} we show the normalization factor formula 
of the bosonization of singular vectors.
\S \ref{sec:norm} gives the proof of the main theorem.
After defining the Whittaker vector for Virasoro algebra,
we state in Theorem \ref{thm:main}
the conjecture of Awata and Yamada, whose proof is the main result of this paper.

In the appendix we will describe the embedding diagrams of Verma modules 
for $\vqt$.
The result is precisely similar to the (non-deformed) Virasoro Lie algebra case,
but we attached this appendix since there seems no explicit description 
in the literature.

\subsection*{Notation}

We follow \cite{M:1995} for the notations of partitions and symmetric functions.
By a partition we mean a non-decreasing finite sequence of positive integers.
For a partition $\lambda$, its total sum is denoted by $|\lambda|$ 
and its length is denoted by $\ell(\lambda)$.
$\lambda \vdash n$ means that $\lambda$ is a partition with the condition 
$|\lambda| = n$.
We sometimes use abbreviated symbols for partitions with large multiplicities,
such as $(1^n) = (1,1,\ldots,1)$.

\subsection*{Acknowledgements}

The author is supported by the Grant-in-aid for 
Scientific Research (No.\ 25800014), JSPS.

This work is also partially supported by the 
JSPS Strategic Young Researcher 
Overseas Visits Program for Accelerating Brain Circulation
``Deepening and Evolution of Mathematics and Physics,
Building of International Network Hub based on OCAMI''.

A part of this work was done while the author 
stayed National Research University  Higher School of Economics 
(winter 2013-14).
The author would like to thank the institute for support and hospitality.

Finally the author would like to thank Professor Ivan Cherednik 
and Professor Masato Taki for showing their interest.

\section{Basics of the deformed Virasoro algebra}
\label{sec:basics}

\subsection{Definition}
\label{ssec:dfn}

Let us begin with the introduction of the deformed Virasoro algebra \cite{SKAO}.
Let $q$ and $t$ be indeterminates and set $\bbF := \bbQ(q,t)$.
Consider the associative algebra over $\bbF$ 
generated by $\{T_n \mid n \in \bbZ\}$ and $1$ 
with the relations
\begin{align}
\label{eq:qvir:rel}
 \sum_{l \ge 0} f_l (T_{m-l} T_{n+l} - T_{n-l} T_{m+l})
=-\dfrac{(1-q)(1-t^{-1})}{1-q/t}
  \bigl((q/t)^m - (q/t)^{-m}\bigr) \delta_{m+n,0}
\end{align}
for any $m,n \in \bbZ$.
Here the coefficients $f_l$'s are given by the generating function
\begin{align}\label{eq:fl}
  \sum_{l \ge 0} f_l z^l
= \exp\Bigl(
   \sum_{n\ge1} 
    \dfrac{1}{n} \dfrac{(1-q^n)(1-t^{-n})}{1+(q/t)^{n}} z^n 
  \Bigr).
\end{align}
This associative algebra will be denoted by $\vqt$.

The algebra $\vqt$ has an automorphism defined by $T_n \mapsto -T_n$ ($n\in\bbZ$).
There is also an isomorphism of algebras defined by 
$$
 \vqt \xrightarrow{\ \sim \ } \vir_{t^{-1},q},\qquad 
 T_n \longmapsto T_n \ (n\in\bbZ).
$$

By the defining relation \eqref{eq:qvir:rel},
$\vqt$ has a $\bbZ$-grading.
Introduce the outer grading operator $D$ 
by 
$$[D,T_n]=n T_n, \quad  [D,1] =0.
$$
Set the $d$-grading subspace by 
$\vir_{q.t}^{(d)} := \{ X \in \vir_{q.t} \mid [D,X] = d\}$.
Then we have a decomposition
$$
 \vqt = \bigoplus_{d \in \bbZ} \vqt^{(d)}
$$
as an $\bbF$-vector space.
Let us also denote by $\vqt^{(+)}$ (resp. $\vqt^{(-)}$) 
the subalgebra of $\vqt$
generated by $\{T_n \mid n >0\}$ (resp. $\{T_n \mid n < 0\}$).
These subalgebras enjoy the decomposition 
$$
  \vir_{q,t}^{(+)} = \bigoplus_{d > 0} \vir_{q.t}^{(d)},\qquad
  \vir_{q,t}^{(-)} = \bigoplus_{d < 0} \vir_{q.t}^{(d)}.
$$

Although $\vir_{q,t}$ is not a Lie algebra (nor a vertex algebra),
we have the following Poincar\'{e}-Birkoff-Witt type theorem,

\begin{lem}\label{lem:pbw}
Each $d$-grading subspace $\vir_{q.t}^{(d)}$ 
has the following $\bbF$-basis.
$$
 \{ T_{n_1} T_{n_2} \cdots T_{n_\ell} \mid
    n_1 \le n_2 \le \cdots \le n_\ell, \ \sum_i n_i = d\}.
$$
\end{lem}
%

\subsection{Degeneration to the Virasoro Lie algebra}
\label{subsec:degenerate}

Let us recall the degeneration of $\vqt$ to the Virasoro Lie algebra,
which is the origin of the name ``deformed Virasoro''.

Let $\ve_1$, $\ve_2$ and $\hbar$ indeterminates and set 
\begin{align}\label{eq:qtve}
q = e^{\hbar \ve_1},\quad  t=e^{\hbar \ve_2}. 
\end{align}
Let us consider the algebra $\vir_{e^{\hbar \ve_1},e^{\hbar \ve_2}}$ as defined over 
$\bbF = \bbQ(e^{\hbar \ve_1}, e^{\hbar \ve_2})$,
and let us take the $\hbar$-expansion. 

\begin{lem}
The generator $T_n$ of $\vqt$ can be expanded as 
\begin{align}\label{eq:expand}
 T_n  = 2 \delta_{n,0} + 
  \ve_1 \ve_2
  \Bigl( L_n  + \dfrac{1}{4}(\ve_1-\ve_2)^2 \delta_{n,0} \Bigr) 
  \hbar^2 +  O(\hbar^4),
\end{align}
where $L_n$'s satisfy
\begin{align}\label{eq:degeneration}
 [L_m,L_n] = (m-n) L_{m+n} + c \dfrac{m^3-m}{12}\delta_{m+n,0}
\end{align}
with
\begin{align}\label{eq:beta}
 c=13-6\bigl(\beta +\beta^{-1}\bigr),\quad 
 \beta := \dfrac{\ve_2}{\ve_1}.
\end{align}
\end{lem}

The relation \eqref{eq:degeneration}
is nothing but the defining relation of the Virasoro Lie algebra.
Although this statement is given in \cite{SKAO},
there seems no proof in the literature, 
so let us write down a derivation for readers' convenience.


\begin{proof}
First we want to obtain an expansion of $f_l$.
In the definition \eqref{eq:fl}, we have
$$
 \dfrac{(1-q^n)(1-t^{-n})}{1+(q/t)^n} 
 = -\dfrac{n^2}{2}\ve_1\ve_2 \hbar^2 
   +\dfrac{n^4}{24}\ve_1\ve_2(\ve_1^2-3\ve_1\ve_2+\ve_2^2) \hbar^4 + O(\hbar^6).
$$
Then $f_l$ has the following expansion.
\begin{align}\label{eq:expand:fl}
f_l = \sum_{d\ge0} f_l^{(d)}\hbar^{2d},\qquad
f_l^{(0)} = \delta_{l,0},\quad
f_l^{(1)} = -\dfrac{l}{2}\ve_1\ve_2,\quad
f_l^{(2)} = \dfrac{l^3}{24}\ve_1\ve_2(\ve_1^2-3\ve_1\ve_2+\ve_2^2)+\dfrac{l^3-l}{48}\ve_1^2\ve_2^2.
\end{align}
On the other hand,
the right hand side of the defining relation \eqref{eq:qvir:rel} of $\vqt$ 
can be expanded as 
\begin{align}\label{eq:expand:rhs}
&\dfrac{(1-q)(1-t^{-1})}{1-q/t}\bigl((q/t)^m-(q/t)^{-m}\bigr) 
= \sum_{d \ge 0}\RHS^{(d)} \hbar^{2d},
\end{align}
with
\begin{align*}
\RHS^{(0)} = 0,\quad
\RHS^{(1)} = 2m \ve_1\ve_2,\quad 
\RHS^{(2)} = \dfrac{m}{6} \ve_1\ve_2\bigl( 2m^2(\ve_1^2+\ve_2^2) 
                      + (1-4m^2)\ve_1\ve_2 \bigr).
\end{align*}
By \eqref{eq:expand:fl} and \eqref{eq:expand:rhs},
we see that $T_n$'s can be expanded as 
$T_n = \sum_{d \ge 0} T_n^{(d)} \hbar^{(d)}$.
Then the left hand side of \eqref{eq:qvir:rel} is expanded as 
\begin{align*}
 \sum_{l \ge 0} f_l (T_{m-l} T_{n+l} - T_{n-l} T_{m+l})
=\sum_{d\ge0} \LHS^{(d)} \hbar^{2d}
\end{align*}
with
\begin{align*}
\LHS^{(0)}
 = &\sum_{l\ge0}f_l^{(0)}(T_{m-l}^{(0)}T_{n+l}^{(0)}-T_{n-l}^{(0)}T_{m+l}^{(0)}),
\\
\LHS^{(1)}
 = &\sum_{l\ge0}f_l^{(0)}(T_{m-l}^{(1)}T_{n+l}^{(0)}+T_{m-l}^{(0)}T_{n+l}^{(1)}
                         -T_{n-l}^{(1)}T_{m+l}^{(0)}-T_{n-l}^{(0)}T_{m+l}^{(1)})
  + \sum_{l\ge0}f_l^{(1)}(T_{m-l}^{(0)}T_{n+l}^{(0)}-T_{n-l}^{(0)}T_{m+l}^{(0)}),
\\
\LHS^{(2)}
 = &\sum_{l\ge0}f_l^{(0)}(T_{m-l}^{(2)}T_{n+l}^{(0)}+T_{m-l}^{(1)}T_{n+l}^{(1)}
                         +T_{m-l}^{(0)}T_{n+l}^{(2)}
                         -T_{n-l}^{(2)}T_{m+l}^{(0)}-T_{n-l}^{(1)}T_{m+l}^{(1)}
                         -T_{n-l}^{(0)}T_{m+l}^{(2)})
\\
 + &\sum_{l\ge0}f_l^{(1)}(T_{m-l}^{(1)}T_{n+l}^{(0)}+T_{m-l}^{(0)}T_{n+l}^{(1)}
                         -T_{n-l}^{(1)}T_{m+l}^{(0)}-T_{n-l}^{(0)}T_{m+l}^{(1)})
  + \sum_{l\ge0}f_l^{(2)}(T_{m-l}^{(0)}T_{n+l}^{(0)}-T_{n-l}^{(0)}T_{m+l}^{(0)}).
\end{align*}
Now the comparison of $\LHS^{(0)}$ and $\RHS^{(0)}$ gives 
$$
 T_{m}^{(0)}T_{n}^{(0)} = T_{n}^{(0)}T_{m}^{(0)} \quad \forall\,m,n \in \bbZ.
$$
Then by $\LHS^{(1)}$, $\RHS^{(1)}$ and the formula $f_l^{(1)}$ in \eqref{eq:expand:fl}, 
we have
$$
  T_{m}^{(1)}T_{n}^{(0)}+T_{m}^{(0)}T_{n}^{(1)}
- T_{n}^{(1)}T_{m}^{(0)}-T_{n}^{(0)}T_{m}^{(1)}
- \dfrac{1}{2}\ve_1\ve_2 \sum_{l\ge0}l(T_{m-l}^{(0)}T_{n+l}^{(0)}-T_{n-l}^{(0)}T_{m+l}^{(0)})
= 2\ve_1\ve_2 \delta_{m+n,0}.
$$
The following conditions imply the above relation.
$$
 T_{n}^{(0)} = 2\delta_{n,0}.
$$
By the obtained conditions, $\LHS^{(2)}$ gets the following simple form.
$$
 \LHS^{(2)} 
 = [T_m^{(1)},T_n^{(1)}]-\ve_1\ve_2(m-n)T_{m+n}^{(1)} 
   + 4f_{|m|}^{(2)} \sign(m) \delta_{m+n,0}.
$$
Here $\sign(m)$ is the sign of the integer $m$.
By this formula, $\RHS^{(2)}$ and $f_l^{(2)}$, we finally have
$$
 [T_m^{(1)},T_n^{(1)}]
 = \ve_1\ve_2 (m-n) T_{m+n}^{(1)} 
   -\ve_1\ve_2 \Bigl(\dfrac{m}{2}(\ve_1-\ve_2)^2 
                     + \dfrac{m^3-m}{12} (6\ve_1^2-13\ve_1\ve_2+6\ve_2^2) \Bigr)
    \delta_{m+n,0}.
$$
This relation can be rewritten as 
$$
 [\widetilde{T}_m^{(1)},\widetilde{T}_n^{(1)}]
 = \ve_1\ve_2 (m-n) \widetilde{T}_{m+n}^{(1)} 
 + \ve_1\ve_2\dfrac{m^3-m}{12} (13\ve_1\ve_2-6\ve_1^2-6\ve_2^2) \delta_{m+n,0}, 
$$
where we put 
$\widetilde{T}_n^{(1)} := T_n^{(1)} - \ve_1\ve_2(\ve_1-\ve_2)^2 \delta_{n,0}/4$.
Thus 
$$
 L_n := \dfrac{1}{\ve_1\ve_2} \widetilde{T}_n^{(1)} 
      = \dfrac{1}{\ve_1\ve_2} T_n^{(1)} - \dfrac{1}{4}(\ve_1-\ve_2)^2 \delta_{n,0}
$$
satisfies the desired relation \eqref{eq:degeneration}.
\end{proof}

\section{Verma modules}
\label{sec:verma}

\subsection{Generalities}

We introduce an analog of the category $\calO$ of the 
representations of Lie algebra.
Since $\vqt$ has no abelian subalgebra,
we use the degeneration functor $\Phi$ \eqref{eq:functor_phi} 
to obtain a useful notion.

\subsubsection{}

Let us briefly recall some important categories of representations 
of the Virasoro Lie algebra, following \cite[Chap. 1, \S 2]{IK}.
Let $\vir$ be the Virasoro Lie algebra over $\bbQ$ with 
the generators $\{L_n \mid n \in \bbZ\}$ and 
the central generator $C$  satisfying the defining relation
$$
 [L_m,L_n] = (m-n)L_{m+n} + C \dfrac{m^3-m}{12}\delta_{m+n,0}.
$$
Let $\frkh$ be the abelian Lie subalgebra of $\vir$ generated by $L_0$ and $C$.
$\vir$ is a $\bbZ$-graded Lie algebra with 
$$
 \vir = \bigoplus_{n\in\bbZ}\vir^{(n)}, \qquad
 \vir^{(n)} := \begin{cases}\bbQ L_n & n \neq 0 \\ \frkh & n=0 \end{cases}.
$$
In this $\bbZ$-grading structure, 
we have a unique $\lambda_n \in \frkh^*$ for each $n \in \bbZ$ such that 
$$
 [h,x] = \lambda_{n}(h) x \quad \forall \, \lambda \in \frkh,\ \forall \, x \in \vir^{(n)}.
$$
Let $\iota$ be the homomorphism defined by 
$$
 \iota: \bbZ \longrightarrow  \frkh^*, \quad n \longmapsto \lambda_{n}.
$$

\begin{dfn} Let $M$ be a $\vir$-module.
\begin{enumerate}
\item
$M$ is called $\bbZ$-graded if 
$M = \oplus_{n \in \bbZ} M^{(n)}$ as a $\bbQ$-vector space 
and 
$$
 \vir^{(m)}.M^{(n)} \subset M^{(m+n)}
$$
for any $m,n \in \bbZ$.
\item 
$M$ is called $\frkh$-diagonalizable if 
$$
 M = \bigoplus_{\lambda \in \frkh^*} M_{\lambda}
$$ 
with $M_\lambda := \{v \in M \mid h.v = \lambda(h) v \ \forall \, h \in \frkh\}$.
\item
$M$ is called $\frkh$-semi-simple if 
$$
 \dim M_{\lambda} < \infty
$$ 
for any $\lambda \in \frkh^*$.
\item
$M$ is called admissible 
if it is $\bbZ$-graded and $\frkh$-diagonalizable.
\end{enumerate}
\end{dfn}

Denote by $\calC'$ 
the category of $\bbZ$-graded $\frkh$-diagonalizable $\vir$-modules.
(The morphisms is defined to be $\vir$-homomorphism respecting the $\bbZ$-gradings.)
$\calC'$  is an abelian category.
Denote by $\calC'_{\text{adm}}$ 
the full subcategory of $\widetilde{C}$ 
consisting of admissible $\vir$-modules.
It is also an abelian category.

\begin{dfn}
The category $\calO'$ for $\vir$ is the full subcategory of
$\calC'_{\text{adm}}$
whose objects consist of $M$ such that there exist
finitely many $\lambda_i \in \frkh^*$ satisfying
$$
 \{\lambda \in \frkh^* \mid M_\lambda \neq 0 \} 
\subset \bigcup_i \bigl(\lambda_i - \iota(\bbZ_{\ge0})\bigr).
$$
\end{dfn}

The category $\calO'$ is nothing but the BGG (Beilinson-Gelfand-Gelfand) category.

\subsubsection{}
Next we turn to the introduction of categories of $\vqt$-representations.
In this paper, a $\vqt$-module means a $\bbQ$-vector space $M$ with a left $\vqt$-action.
We denote by `$.$' the action of $\vqt$.
For example, $X.v$ means the action of $X$ of $\vqt$ on $v \in M$.

A $\vqt$-module is called $\bbZ$-graded 
if $M = \oplus_{n \in \bbZ} M^{(n)}$ as a $\bbQ$-vector space and 
$\vqt^{(m)}.M^{(n)} \subset M^{(n+m)}$.

\begin{dfn}
Denote by $\calC$ the category of $\vqt$-modules with $\bbZ$-grading,
where the momorphism is a $\vqt$-homomorphism respecting $\bbZ$-gradings.
\end{dfn}

$\calC$ is an abelian category.

The degeneraton procedure explained in \S \ref{subsec:degenerate} 
induces a functor $\Phi$ from
the category of $\vqt$-modules 
to the category of representations of $\vir$.
On an element $v$ of a $\vqt$-module $M$,
$\Phi$ can be written as
\begin{align}\label{eq:functor_phi}
 \Phi(T_n).v = \lim_{\hbar^2 \to 0} 
  \dfrac{1}{\beta \hbar^2}
  \Bigl(T_n-2-\dfrac{1}{4\beta}(\beta-1)^2\Bigr).v.
\end{align}
Here the parameter $\beta$ is defined by \eqref{eq:beta}.
Note that this functor preserves the $\bbZ$-grading structure of modules.


\begin{dfn}
\begin{enumerate}
\item
An object $M \in \Ob(\calC)$ is called admissible
if $\Phi(M) \in \Ob(\calC')$.
The full subcategory $\calC$ consisting of admissible modules 
is denoted by $\calC_{\text{adm}}$.
\item
The category $\calO$ is the full subcategory of $\calC_{\text{adm}}$
consisting of $M$ satisfying $\Phi(M) \in \Ob(\calO')$.
\end{enumerate}
\end{dfn}

$\calC_{\text{adm}}$ is an abelian category.
The functor $\Phi$ preserves the structure of abelian categories of $\bbZ$-graded modules.

\subsection{Verma module}

\begin{dfn}
$M \in \Ob(\calC)$ is called a highest weight module 
with highest weight $h$ 
if there exists a non-zero element $v \in M$ such that 
\begin{enumerate}
\item  $T_0.v = h v$.
\item  $T_n.v = 0$ for any $n>0$.
\item  $T_{-\lambda}.v$'s ($\lambda$ runs over the set of all partitions) span $M$.
\end{enumerate}
The vector $v$ is called a highest weight vector of $M$.
\end{dfn}

A highest weight module is always an object of $\calC_{\text{adm}}$.

\begin{dfn}
The Verma module $M(h)$ of $\vqt$ 
is defined to be generated by the vector 
$1_{h}$ with the properties 
$$
 T_0.1_{h}=h 1_{h}, \qquad 
 T_n.1_{h}=0 \ (n \ge 1).
$$
\end{dfn}

The Verma module has a $\bbZ$-grading.
Introduce the outer grading operator $D$ by 
\begin{align}\label{eq:opD}
[D,T_n]=n T_n, \qquad  D.1_{h}= h 1_{h}.
\end{align}
Then we have the decomposition 
$$
 M(h) = \bigoplus_{n \in \bbZ_{\ge 0}} M(h)^{(n)},\qquad 
 M(h)^{(n)} := \{v \in M(h) \mid D.v = (h-n) v\}.
$$
By the defining relation of $\vqt$,
$M(h)^{(n)}$ has a basis 
\begin{align}\label{eq:Verma:basis}
\{T_{-\lambda}.1_{h} \mid \lambda \vdash n \}
\end{align}
of $M(h)^{(n)}$,
where we set
$T_{-\lambda} :=  T_{-\lambda_1} T_{-\lambda_2} \cdots T_{-\lambda_\ell}$  
for a partition $\lambda=(\lambda_1,\lambda_2,\ldots,\lambda_\ell)$ of $n$.

The Verma module $M(h)$ is a highest weight module with highest weight vector $1_h$.
It enjoys the following universal property.

\begin{lem}
\begin{enumerate}
\item 
For any highest weight module $M$ with highest weight $h \in \bbC$,
there exists a surjective homomorphism $M(h) \twoheadrightarrow M$.
\item
The Verma module $M(h)$ has a unique maximal proper $\bbZ$-graded submodule 
$N(h) \in \Ob(\calC)$.
It is also the maximal proper submodle of $M(h)$.
\end{enumerate}
\end{lem}

The proof is by a standard argument.

The quotient module is denoted by
$$
 L(h) := M(h) / N(h).
$$
It is an irreducible highest weight module with highest weight $h$.
Thus $L(h) \in \Ob(\calC_{\text{adm}})$.

We also have 
\begin{lem}
The image $\Phi(M(h))$ of Verma module under the functor $\Phi$ \eqref{eq:functor_phi}
is again a Verma module of $\vir$.
\end{lem}

Now recall
\begin{fct}[{\cite[Lemma 1.2]{IK}}]
The irreducible quotients of the Verma modules 
of $\vir$ exhaust the simple objects of $\widetilde{\calO}$.
\end{fct}

Therefore we have
\begin{lem}
$\{L(h) \mid h\}$ exhausts the simple objects of $\calO$.
\end{lem}

\subsection{Shapovalov form}

Now we recall the contravariant form (an analogue of the Shapovalov form for Lie algebras) 
on the Verma module and the Kac determinant formula for $\vir_{q,t}$.
These were introduced in \cite{SKAO} in comparison with the Virasoro Lie algebra.

Introduce the anti-involution
$$
 \sigma: \vir_{q,t} \to \vir_{q,t},\qquad
 T_n \longmapsto T_{-n}
$$
over $\bbF$.
By the Poincar\'{e}-Birkoff-Witt theorem (Lemma \ref{lem:pbw}),
we have the direct sum decomposition
$$
 \vir_{q,t} = \vir_{q,t}^{(0)} 
  \oplus \bigl\{ \vir_{q,t}^{(-)} \vir_{q,t} + \vir_{q,t} \vir_{q,t}^{(+)} \bigr\}. 
$$
as an $\bbF$-vector space.
In this decomposition, denote the projection to the first component by
$$
 \pi: \vir_{q,t} \longrightarrow \vir_{q,t}^{(0)}. 
$$

\begin{dfn}\label{dfn:Shapovalov}
\begin{enumerate}
\item
Define the contravariant bilinear form 
$\calF: \vir \otimes \vir \to \vir^{(0)}$ to be
$$
 \calF(x,y) = \pi(\sigma(x)y)
$$
for $x,y \in \vir$.

\item
The contravariant form 
$\langle -,-\rangle : M(h) \otimes M(h) \to \bbF$ is determined by
$$
 \langle x.1_{h},y.1_{h} \rangle 1_{h} = 
 \sigma(x)y.1_{h}
$$
for $x,y \in \vir$.
We call it the Shapovalov form for $M(h)$.
\end{enumerate}
\end{dfn}

Obviously we have $\langle x.1_{h},y.1_{h} \rangle 1_{h} = \pi(\sigma(x)y).1_{h}$
As in the case the Shapovalov form of Lie algebras,
this $\langle -,-\rangle$ enjoys the following properties.

\begin{lem}
\begin{enumerate}
\item 
For any $v,w \in M(h)$ we have
$\langle v,w \rangle = \langle w,v \rangle$.
\item
For any $v,w \in M(h)$ and $x \in \vir_{q,t}$ we have
$\langle x.v,w \rangle = \langle v,\sigma(x).w \rangle$.
\item
The radical 
$$
 \Rad_{\langle-,-\rangle} := \{ v \in M(h) \mid \langle v,M(h) \rangle = 0  \}
$$
of the contravariant form is a maximal $\vqt$-submodule of $M(h)$.
\end{enumerate}
\end{lem}

As in the case of the Virasoro Lie algebra,
the determinant of the contravariant form tells us the irreducibility of $M(h)$.
Let us recall the argument briefly.
By (2) of the above Lemma, we have
$\pair{ M(h)^{(m)} }{ M(h)^{(n)} } = 0$ for $m \neq n$.
Then by (3), we see that $M(h)$ is irreducible if and only if 
$\Rad_{\langle-,-\rangle}  \cap \, M(h)^{(n)}$ is trivial for any $n \in \bbZ_{\ge 0}$.
The latter is equivalent to the vanishing of the determinant 
of the matrix whose elements are given by 
of the form $\langle -,-\rangle|_{ M(h)^{(n)}\otimes M(h)^{(n)}}$
evaluated on some basis of $M(h)^{(n)}$.
Recall that we have a basis $\{T_{-\lambda}.1_{h} \mid \lambda \vdash n \}$
\eqref{eq:Verma:basis} of $M(h)^{(n)}$.
So let us set 
$$
 \Det_n(q,t,h) := 
 \det \bigl(\langle T_{-\lambda}.1_{h},T_{-\mu}.1_{h} \rangle \bigr)_{\lambda,\mu \, \vdash n}.
$$
It is an element of $\bbF[h]$.
As conjectured in \cite{SKAO} and proved in \cite{BP},
we have the formula 
\begin{align}\label{eq:Kac}
 \Det_n(q,t,h) = C_n
  \prod_{\substack{r,s \ge 1 \\ r s \le n}} 
   (h^2-h_{r,s}^2)^{p(n-r s)}
   \Bigl(\dfrac{(1-q^r)(1-t^r)}{q^r+t^r}\Bigr)^{p(n-r s)}.
\end{align}
Here $p(m) := \#\{\lambda \mid \lambda \,\vdash m\}$ denotes the partition number of $m$,
$C_n \in \bbQ\setminus\{0\}$ is some constant, and
$$
 h_{r,s} := t^{r/2} q^{-s/2} + t^{-r/2} q^{s/2}.
$$
This is an analogue of the Kac determinant formula
for the Virasoro Lie algebra.

\subsection{Singular vector}

Let us introduce some basic notions.
For a $\vqt$-module $M$,
we set
$$
 M^{\vqtp} := \{ v \in M \mid T_n.v = 0 \quad \forall n \in \bbZ_{>0} \}.
$$

\begin{dfn} 
Let $M=\oplus_{n\in\bbZ}M^{(n)}$ 
be a $\bbZ$-graded $\vqt$-module and $v$ be a non-zero element of $M$.
\begin{enumerate}
\item
 $v$ is called a weight vector if it belongs to a graded component $M^{(d)}$ 
 with some $d \in \bbZ$.
 The integer $d$ is called the weight of the weight vector $v$.
\item
 A weight vector $v$ is called a singular vector if $v \in  M^{\vqt^{+}}$.
\end{enumerate}
\end{dfn}

Similar definitions can be introduced for the Virasoro Lie algebra $\vir$.
We immediately have

\begin{lem}\label{lem:sing:deg}
By the functor $\Phi$ \eqref{eq:functor_phi},
a singular vector in the Verma module $M(h)$ of $\vqt$
is mapped to a singular vector in the Verma module of $\vir$.
\end{lem}

As in the case of Lie algebras, singular vectors play important role 
in the structure of Verma module of $\vqt$.
For instance, to study embedding diagrams of Verma modules,
we need the following uniqueness property of singular vectors. 
Set
$$
 \bigl(M(h)^{(n)}\bigr)^{\vqtp} := M(h)^{\vqtp} \cap M(h)^{(n)}.
$$

\begin{lem}\label{lem:hom<=1}
Then for each $n \in \bbZ_{>0}$ we have
$$
 \dim \bigl(M(h)^{(n)}\bigr)^{\vqtp} \le 1.
$$
In particular, for any $h,h' \in \bbF$ we have
$$
 \dim \Hom_{\vqt}\bigl(M(h),M(h')\bigr) \le 1.
$$
Here $\Hom_{\vqt}$ means the space of homomorphisms 
between $\bbZ$-graded $\vir_{q,t}$-modules.
\end{lem}
\begin{proof}
A similar statement for the Virasoro Lie algebra $\vir$ is known
(see \cite[Proposition 5.1]{IK} for example).
Then our statement holds by Lemma \ref{lem:sing:deg}.
\end{proof}

On the other hand,
for the existence of singular vector,
the Kac determinant \eqref{eq:Kac} must vanish.
The existence of singular vector 
implies the existence of homomorphism between Verma modules,
as in the case of Lie algebras.
Such a homomorphism is induced by a special element of $\vqt$, 
which we will call Shapovalov element following the Lie algebra case. 

In the lemma below, assume $q,t,h$ are indeterminates and 
let $\calV_{r,s}$ be the algebraic set in $\bbC[q,t,h]$ defined by $h^2-h_{r,s}^2 = 0$. 
Let $\bbF_{r,s} := \bbC[\calV_{r,s}]$ be the coordinate ring of $\calV_{r,s}$,
and $I_{r,s} := I(\calV_{r,s})$ the ideal of $\calC_{r,s}$ in $\bbC[q,t,h]$.
We also denote by $\pi_{r,s} := \bbC[q,t,h] \to \bbF_{r,s}$ the natural homomorphism.
Consider $\vqt$ as an object defined over $\bbF=\bbQ(q,t)$ 
and $M(h)$ defined over $\bbF[h]$.

\begin{lem}\label{lem:sing:exist}
Fix a pair $(r,s)$ of positive integers, and set $n:= r s$.
Then there exists $S_{r,s} \in \vqt^{(-n)}$, called a Shapovalov element,
satisfying
\begin{enumerate}
\item 
$T_k S_{r,s}.1_h \in I_{r,s} \otimes_{\bbF} M(h)$ for any $k \in \bbZ_{>0}$.
\item
$S_{r,s} =\sum_{\lambda \, \vdash n} H_{\lambda} T_{-\lambda}$ with 
$H_{\lambda} \in \bbF[h]$ and $H_{(1^n)} = 1$.
\end{enumerate}
\end{lem}

\begin{proof}
We mimic the argument in \cite[Proposition 5.2]{IK}.
We first show the existence of an element 
$$
 \widetilde{S}_{r,s} =\sum_{\lambda \, \vdash n} c_{\lambda} T_{-\lambda} 
 \in \bbF_{r,s} \otimes_{\bbF}\vir^{(-n)} 
$$
with $c_{(1^n)} =1$ and
\begin{align}\label{eq:lem:shap_elem}
 \widetilde{S}_{r,s}.1_{h} \in \bigl(M(h)^{(-n)}\bigr)^{\vqtp}.1_{h}\quad
 \forall (q,t,h) \in \calV_{r,s}.
\end{align}
The condition \eqref{eq:lem:shap_elem} is a system of linear equations in $\{c_{\lambda}\}$
defined over $\bbF_{r,s}$.
Although any solution lies in the quotient field of $\bbF_{r,s}$,
it is enough to consider solutions in $\bbF_{r,s}$
by multiplying non-zero elements of $\bbF_{r,s}$.
Thus we want to show the existence of solution on a Zariski dense subset of $\bbF_{r,s}$.
The set 
$$
 \calD_{r,s} := \calV_{r,s} \cap 
  \bigcup_{\substack{\alpha,\beta\in\bbZ_{>0}  \\ \alpha \beta < n}} 
  \calV_{\alpha,\beta} 
$$
is a Zariski closed subset of $\calV_{r,s}$.
For $(q,t,h) \in \calV_{r,s} \setminus \calD_{r,s}$, 
we have $\det_m \neq 0$ for any $m <n$ and $\det_n = 0$.
Hence \eqref{eq:lem:shap_elem} has a solution on a Zariski dense subset of $\bbF_{r,s}$.
By Lemma \ref{lem:hom<=1}, a solution is uniquely determined by $c_{(1^n)}$.
Thus we may assume $c_{(1^n)}=1$.

Now choose $H_\lambda \in \bbC[q,t,h]$ such that $\pi_{r,s}(H_\lambda) = c_\lambda$,
and define $S$ by the second condition in the statement.
Since $\Ker \pi_{r,s} = I_{r,s}$, we have the desired result.
\end{proof}

By Lemma \ref{lem:hom<=1} and Lemma \ref{lem:sing:exist}, we have

\begin{prop}\label{prop:hom=1}
Assume $|q|,|t| \neq 1$ and $h=h_{r,s}$ with $r,s \in \bbZ_{>0}$.
Then 
$$
 \dim \Hom_{\vqt}(M(h_{r,s}+r s)[r s], M(h_{r,s})) = 1.
$$
Here $M(h_{r,s}+r s)[r s]$ is the $\bbZ$-graded $\vir_{q,t}$-module 
obtained from $M(h_{r,s}+r s)$ with the $\bbZ$-grading shifted by $r s$.
Moreover, such a homomorphism is a scaler multiple of the embedding 
which maps a highest weight vector $1_{h_{r,s}+r s}$ to $S_{r,s}.1_{h_{r,s}}$.
\end{prop}

\section{Normalization factor of singular vector}
\label{sec:sing}

\subsection{Bosonization}

Let us recall the bosonization of $\vir_{q,t}$ following \cite{SKAO}.
Consider the Heisenberg Lie algebra $\calH$ over $\bbF$
generated by $\{a_n \mid n \in \bbZ \setminus \{0\} \}$ and $1$ 
with the commutation relations
$$
 [a_m,a_n] = m \dfrac{1-q^{|m|}}{1-t^{|m|}}\delta_{m+n,0}.
$$
Denote by $\calF$ the Fock module of $\calH$.
It is generated by $1_{\calF}$ satisfying
$$
 a_n.1_{\calF} = 0 \ (n > 0).
$$
For an indeterminate (or an element of some ring) $\alpha$,
set 
$$
 \bbF_\alpha:=\bbF[q^{\pm1/2},t^{\pm 1/2},q^{\pm\alpha}]
$$ 
and $\calF_{\alpha}:=\calF\otimes_{\bbF} \bbF_\alpha$.
Let us denote by $1_{\calF,\alpha} := 1_{\calF} \otimes 1 \in \calF_{\alpha}$.
We have a grading 
$$
 \calF_{\alpha} = \bigoplus_{n \in \bbZ_{\ge0}}\calF^{(-n)}_{\alpha}
$$
by setting $\deg a_n = n$.

\begin{fct}[{\cite{SKAO}}]
$\vir_{q,t}$ acts on $\calF_{\alpha}$  
if $T(z) = \sum_{n \in \bbZ} T_n z^{-n}$ is set to be
\begin{align}
\label{eq:qvir:bosonization}
& T(z) = \Lambda^{+}\bigl( (q/t)^{-1/2}z \bigr) \cdot 
          (q/t)^{1/2}q^{\alpha}
       + \Lambda^{-}\bigl( (q/t)^{ 1/2}z \bigr) \cdot 
          (q/t)^{-1/2}q^{-\alpha},
\end{align}
with
\begin{align*}
& \Lambda^{\pm}(z) 
  := \exp\Bigl(
         \pm \sum_{n\ge1}
            \dfrac{1-t^{-n}}{1+(q/t)^n} \dfrac{a_{-n}}{n} z^{n}
         \Bigr) 
     \exp\Bigl(
         \mp \sum_{n\ge1}
            (1-t^n) \dfrac{a_{n}}{n} z^{-n}
         \Bigr).
\end{align*}
\end{fct}
One can prove that this formula does define an action of $\vir_{q,t}$ 
by using the formulas
\begin{align*}
&\Lambda^{\pm}(z)\Lambda^{\pm}(w)
 =f(w/z) \cdot \nord{\Lambda^{\pm}(z)\Lambda^{\pm}(w)},\\
&\Lambda^{\pm}(z)\Lambda^{\mp}(w)
 =f(w/z)^{-1} \cdot \nord{\Lambda^{\pm}(z)\Lambda^{\mp}(w)}.
\end{align*}
Note that
$$
T_0.1_{\calF,\alpha} = 
 (q/t)^{1/2}q^\alpha +  (q/t)^{-1/2}q^{-\alpha}
$$

Since we take $q$ and $t$ to be indeterminates, 
the map 
\begin{align}
\label{eq:map:M-F}
M(h) \longrightarrow \calF_{\alpha}
\end{align}
induced by the bosonization \eqref{eq:qvir:bosonization} 
and the correspondence 
$1_{M,h} \mapsto 1_{\calF,\alpha}$
with 
\begin{align}
\label{eq:qvir:boson:hw}
 h= (q/t)^{1/2}q^\alpha +  (q/t)^{-1/2}q^{-\alpha}
\end{align}
gives an isomorphism of $\vir_{q,t}$-modules.
It also respects the gradings: 
$M(h)^{(-n)} \to \calF^{(n)}_{\alpha}$.
If we specialize $q$ and $t$ to generic complex numbers,
the above map is still an isomorphism between irreducible modules.
The non-generic values are 
given by the zeros of the ($\vir_{q,t}$ version of) Kac determinant
\eqref{eq:Kac}.

The Fock space $\calF_{\alpha}$ is isomorphic to 
the space of symmetric functions.
Let us denote by $\Lambda$ 
the space of symmetric functions defined over $\bbZ$.
If we take $\{x_i\}$ the set of infinite variables,
then we may identify 
\begin{align}\label{eq:Lambda:sym}
\Lambda = \bbZ[x_1,x_2,\ldots]^{\mathfrak{S}_{\infty}}.
\end{align}
It has an natural grading 
$$
 \Lambda = \bigoplus_{n\in\bbZ_{\ge 0}} \Lambda^{(n)}.
$$
For a ring $R$, Set $\Lambda_{R} := \Lambda \otimes_{\bbZ} R$.
Denote by $p_r$ the $r$-th power symmetric function.
Under the identification \eqref{eq:Lambda:sym},
we have 
$$
 p_r = \sum_i x_i^r.
$$
The power symmetric functions give rise to a basis 
$\{p_\lambda \mid \lambda \vdash n\}$
of $\Lambda^{(n)}_{\bbQ}$, where we set 
$$
 p_\lambda := p_{\lambda_1}p_{\lambda_2}\cdots p_{\lambda_k}.
$$
for a partition $\lambda=(\lambda_1,\lambda_2,\ldots,\lambda_k)$.
Now we introduce the isomorphism of $\bbF_\alpha$-vector spaces by
\begin{align}
\label{eq:isom:F-L}
\calF_\alpha \longrightarrow \Lambda_{\bbF_\alpha},\quad
\alpha_{-\lambda}.1_{\calF,\alpha} \longmapsto 
p_\lambda.
\end{align}
Here we used the notation 
$a_{-\lambda} := a_{-\lambda_1}a_{-\lambda_2}\cdots$ 
for a partition $\lambda$.
The set $\{a_{-\lambda}.1_{\calF,\alpha}\mid \lambda \vdash n \}$ 
is obviously an $\bbF_\alpha$-basis of $\calF_{\alpha}^{(n)}$,
so that the map \eqref{eq:isom:F-L} is indeed an isomorphism 
of graded $\bbF$-vector spaces.

Hereafter we denote the composition of the maps \eqref{eq:map:M-F} 
and \eqref{eq:isom:F-L} by $\iota_h$:
\begin{align}
\label{eq:Fock}
\iota_{h}: M(h) \longrightarrow \Lambda_{\bbF_\alpha},
\end{align}
where $h$ and $\alpha$ is related as \eqref{eq:qvir:boson:hw}.
Then we immediately have $\iota_{h}(1_h) = 1$.

\subsection{Bosonization formula of singular vector}
Let us recall the bosonization formula of singular vector
given in \cite{SKAO}.
By the argument in the previous section,
in particular the Kac formula \eqref{eq:Kac},
a singular vector exists if and only if 
the highest weight $h$ of the Verma module $M(h)$ is of the form
$h=h_{r,s}:= t^{r/2}q^{-s/2}+t^{-r/2}q^{s/2}$ 
with some $r,s \in \bbZ_{\ge1}$.

Denote by $v_{r,s}$ the singular vector of $M(h_{r,s})$.
It is determined uniquely up to scalar multiplication.

\begin{fct}[{\cite{SKAO}}]\label{fct:SKAO:propto}
Under the bosonization map \eqref{eq:Fock}, 
we have
$$
 \iota_{h_{r,s}}(v_{r,s})
 =  B_{r,s}(q,t) J_{(s^r)}(q,t)
$$
with some $B_{r,s}(q,q^{\alpha},t) \in \bbF_\alpha$.
Here we denote by $J_\lambda(q,t)$ 
the integral form of the Macdonald symmetric function 
for a partition $\lambda$.
\end{fct}

Let us briefly recall the Macdonald symmetric function.
For a complete description, see \cite[Chap. VI]{M:1995}.
On the space $\Lambda_{\bbF}$ of symmetric function,
introduce the pairing $\langle \cdot,\cdot\rangle_{q,t}$ by
$$
 \langle p_{\lambda},p_{\mu} \rangle_{q,t}
 := \delta_{\lambda,\mu} z_\lambda(q,t)
$$
with
\begin{align}
\label{eq:z_lambda}
z_\lambda(q,t) := 
 z_\lambda \cdot 
 \prod_{i=1}^{\ell(\lambda)}\dfrac{1-q^{\lambda_i}}{1-t^{\lambda_i}},
\quad
&z_{\lambda}:=
 \prod_{n\ge1}n^{m_n(\lambda)} m_n(\lambda) !.
\end{align}
Here $m_n(\lambda) := \{ 1\le i \le \ell(\lambda) \mid \lambda_i = n\}$ 
is the multiplicity of $n$ in $\lambda$.
The Macdonald symmetric functions $\{P_\lambda(q,t) \mid \lambda\vdash n\}$ 
is characterized by the following two properties:
\begin{itemize}
\item 
 $\{P_\lambda(q,t) \mid \lambda\vdash n\}$ is an orthogonal basis 
 of $\Lambda_\bbF$ with respect to the pairing $\pair{\cdot}{\cdot}_{q,t}$:
 $$ 
  \langle P_\lambda(q,t),P_\mu(q,t) \rangle_{q,t} = 0 
  \text{ unless } \lambda = \mu.
 $$
\item
 $P_\lambda(q,t)$ has an expansion 
 $$
  P_\lambda(q,t) = m_\lambda + 
                   \sum_{\mu < \lambda} M(P,m)_{\lambda,\mu} m_\mu 
 $$
 with respect to the monomial symmetric functions $m_{\mu}$.
 Here $M(P,m)_{\lambda,\mu} \in \bbF$ and 
 $\mu \le \lambda$ means the dominance ordering: 
 $\sum_{1\le i \le j} \mu_i \le \sum_{1\le i \le j} \lambda_i$ 
 for any positive integer $j$.
\end{itemize}
The integral form $J_{\lambda}(q,t)$ of the Macdonald symmetric function 
is given by 
$$
 J_{\lambda}(q,t) := P_{\lambda}(q,t) \cdot 
  \prod_{\square\in\lambda}(1-q^{a_\lambda(\square)}t^{l_\lambda(\square)+1}).
$$
Here $a_\lambda(\square)$ and $l_\lambda(\square)$  are 
the arm  and the leg of the box $\square$ located at 
$(i,j)\in\bbZ^2_{\ge 1}$ with respect to $\lambda$,
which are by 
\begin{align}
a_\lambda(\square)    := \lambda_i-j,\quad 
\ell_\lambda(\square) := \lambda'_j-i. \label{eq:armleg}
\end{align}
Here $\lambda'$ is the conjugate partition of $\lambda$, 
which is obtained by transposing the Young diagram of $\lambda$. 


Let us normalize the singular vector $v_{r,s}$ by
setting the coefficient of $T_{-1}^{r s}. 1_{h_{r,s}}$ to be one:
\begin{equation}\label{eq:vrs:normalize}
 v_{r,s} = \bigl( T_{-1}^{r s} + 
           \sum_{\lambda \vdash r s,\, \lambda \neq (1^{r s})}
           \chi_\lambda T_{\lambda} \bigr).1_{h_{r,s}},
 \quad
 \chi_\lambda \in \bbF.
\end{equation}
By Fact \ref{fct:SKAO:propto}, we find 
$\iota_{h_{r,s}}(v_{r,s}) \propto J_{(s^r)}(q,t)$.
In \cite[\S4.5, Conjecture 4.68]{S:2003},
the following normalization coefficient is conjectured:

\begin{thm}\label{thm:sing:norm}
$$
 \iota_{h_{r,s}}(v_{r,s}) = J_{(s^r)}(q,t) \cdot 
 \prod_{i=1}^r \prod_{j=1}^s \dfrac{q^j-t^i}{q^j t^i}.
$$
\end{thm}
We will give a proof of this formula in \S\ref{subsect:sing:norm}

\subsection{The case of Virasoro Lie algebra}

In the case of Virasoro Lie algebra,
the singular vector of Verma module 
is identified with Jack symmetric function,
and a similar formula of the normalization coefficient 
in Theorem \ref{thm:sing:norm}
is known and already proved.
Let us review these facts.

Let us recall the definition of singular vectors, 
fixing notations on Virasoro algebra and its Verma module.
The Virasoro algebra $\vir$ is the Lie algebra over $\bbQ$ 
generated by $L_n$ ($n\in\bbZ$) and $C$ (central) with the relation
\begin{align}\label{eq:Vir}
[L_m,L_n]=(m-n)L_{m+n}+\dfrac{C}{12}m(m^2-1)\delta_{m+n,0},\quad
[L_n,C]=0.
\end{align}
Let $c$ and $h'$ be complex numbers (or indeterminates), 
and denote by $M'(c,h')$ the Verma module of $\vir$.
It is generated by the highest weight vector $1_{c,h'}$ satisfying
\begin{align*}
C.1_{c,h'} = c 1_{c,h'},\quad
L_0.1_{c,h'} = h' 1_{c,h'},\quad
L_n.1_{c,h'} = 0 \ (n>0).
\end{align*}
$M'(c,h')$ has an $L_0$-weight space decomposition: 
$M'(c,h')=\bigoplus_{n\in\bbZ_{\ge 0}} M'(c,h')^{(n)}$ 
with
\begin{align*}
M'(c,h')^{(n)}:=\{v\in M'(c,h')\mid L_0 v=(h'-n)v \}.
\end{align*}
A basis of $M'(c,h')^{(-n)}$ is given by
\begin{align*}
\{L_{-\lambda}.1_{c,h'} \mid \lambda \vdash n\},
\end{align*}
with
\begin{align*}
L_{-\lambda} := L_{-\lambda_1} L_{-\lambda_2}\cdots L_{-\lambda_k}
\end{align*}
for the partition 
$\lambda=(\lambda_1,\lambda_2,\ldots,\lambda_k)$.
Now, an element $v$ of $M'(c,h')^{(n)}$ is called a singular vector 
of level $n$ if 
\begin{align*} 
L_k v=0 \ \text{ for any } k\in\bbZ_{>0}.
\end{align*}
The existence condition of singular vector 
is encoded in the Kac determinant formula.

In order to write down the Kac formula,
let us fix some notations of the contravariant form on $\vir$.
Let $\sigma': U(\vir) \to U(\vir)$ be the anti-homomorphism 
determined by $\sigma'(L_n) := L_{-n}$ and $\sigma'(C):=C$.
The contravariant form is the bilinear form 
$$
 \langle -,- \rangle': 
 M'(c,h') \otimes_{\bbQ[c,h']} M'(c,h') \to \bbQ[c,h']
$$ 
determined by 
$$
 \langle x.1_{c,h'},y.1_{c,h'}\rangle' 1_{c,h'} = \sigma'(x)y.1_{c,h'}.
$$
Now define
\begin{align*}
\Det'_n(c,h') 
:= \det\bigl(\langle L_{-\lambda}.1_{c,h'},L_{-\mu}.1_{c,h'} \rangle'
 \bigr)_{\lambda,\mu \, \vdash n}
\end{align*}
for each $n\in\bbZ_{\ge 0}$.
It has the following factored formula,
as conjectured in \cite{K:1979} and shown in \cite{FF1}, \cite{FF2}.
\begin{align}
\label{eq:Kac'}
\Det'_n(c,h')
=\prod_{\lambda \vdash n}2^{\ell(\lambda)}z_\lambda
 \times
 \prod_{\substack{r,s\in\bbZ_{\ge 1}\\ r s\le n}} 
 \bigl(h'-h'_{r,s}\bigr)^{p(n-r s)} .
\end{align}
Here the factor $z_{\lambda} \in \bbZ$ is given by \eqref{eq:z_lambda}.
The zeros $h_{r,s}$ are functions of $c$, and they are given by 
\begin{align}\label{eq:hrst}
h'_{r,s}=h'_{r,s}(\beta) := \dfrac{(r t-s)^2-(\beta-1)^2}{4\beta}
\end{align}
with the parametrisation 
\begin{align}\label{eq:ct}
c=c(\beta):= 13-6(\beta+\beta^{-1}).
\end{align}
Since
$\{h'_{r,s}(\beta) \mid r,s\in\bbZ_{\ge1}, \ r s = n\} = 
 \{h'_{s,r}(\beta^{-1}) \mid r,s\in\bbZ_{\ge1}, \ r s = n\}$,
the set of zeros is independent of the parametrization \eqref{eq:ct}.

Next we introduce the bosonization of the Virasoro Lie algebra.
Consider the Heisenberg algebra $\calH'$ generated by $a'_n$ ($n\in\bbZ$) 
with the relation
\begin{align*}
[a'_m,a'_n]=m \delta_{m+n,0}.
\end{align*}
For a fixed indeterminate (or a complex number) $\rho$, 
consider the correspondence
\begin{align}
\label{eq:vir:boson}
L_n \longmapsto 
 \dfrac{1}{2}\sum_{m\in\bbZ}\nord{a'_m a'_{n-m}}-(n+1)\rho a'_n,
\quad
C \longmapsto 1-12\rho^2,
\end{align}
where the symbol $\nord{\ }$ means the normal ordering of $\calH'$.
This correspondence determines a well-defined morphism 
$U(\vir) \longrightarrow \widehat{U}(\calH')$,
where $\widehat{U}(\calH')$ is 
a completion of the universal enveloping algebra $U(\calH')$. 

Denote by $\calF'_{\alpha'}$ the Fock representation of $\calH'$ 
with the highest weight $\alpha'$.
The highest weight vector is denoted by $1_{\calF',\alpha'}$,
and it satisfies
$$
 a'_0 1_{\calF',\alpha'} = \alpha' a'_0,\quad
 a'_n 1_{\calF',\alpha'} = 0 \ (n>0).
$$
By the map \eqref{eq:vir:boson}, $\calF'_{\alpha'}$ 
is a $\vir$-module.
Then the map 
\begin{align}\label{eq:map:M'-F'}
M'(c,h') \longrightarrow \calF'_{\alpha'}
\end{align}
induced by 
$1_{c,h} \mapsto 1_{\calF',\alpha'}$ 
and \eqref{eq:vir:boson}
gives rise to a homomorphism between $\vir$-modules with 
\begin{align}\label{eq:FF:ch}
 c = 1-12\rho^2, \quad 
 h' = \dfrac{1}{2}\alpha'(\alpha'-2\rho).
\end{align}

Now by the parametrization \eqref{eq:ct}, \eqref{eq:hrst} 
and the correspondence \eqref{eq:FF:ch}, a singular vector occurs at 
\begin{align*}
c=1-12\rho(\beta)^2,\quad 
h'=h'_{r,s}(\beta)=\dfrac{1}{2}\alpha'_{r,s}(\beta)
 \bigl(\alpha'_{r,s}(\beta)-2\rho(\beta)\bigr)
\end{align*} 
with
\begin{align}\label{eq:hw:rs:boson}
\rho(\beta) := \dfrac{1}{\sqrt{2}}(\beta^{1/2}-\beta^{-1/2}),\quad
\alpha'_{r,s}(\beta) :=
 \dfrac{1}{\sqrt{2}}\big((r+1)\beta^{1/2}-(s+1) \beta^{-1/2}\big).
\end{align}

The Fock space $\calF'_{\alpha'}$ is naturally identified 
with $\Lambda_{\bbQ(\beta^{1/2})}$ under the map 
\begin{align}\label{eq:map:F'-L'}
\calF'_\alpha \otimes \bbQ(\beta^{1/2}) \longrightarrow 
 \Lambda_{\bbQ(\beta^{1/2})},\quad 
v \longmapsto 
\langle 1_{\calF',\alpha'},
\exp\Bigl(\sqrt{\dfrac{\beta}{2}}\sum_{n=1}^{\infty}\dfrac{p_n}{n}a_n\Bigr)v
\rangle'.
\end{align}
Under this map,
an element $a_{-\lambda}1_{\calF',\alpha'}$ is mapped to 
$(\beta/2)^{\ell(\lambda)/2} p_\lambda$,
so that this map is indeed an isomorphism of $\bbQ(\beta^{1/2})$-vector spaces. 

Let us denote by $\iota'_{\beta,h'}$ 
the composition of the maps \eqref{eq:map:M'-F'} 
and \eqref{eq:map:F'-L'}:
\begin{align*}
\iota'_{\beta,h'}: M'(c(\beta),h') 
 \longrightarrow \Lambda_{\bbQ(\beta^{1/2})}. 
\end{align*} 
Here we fix $\beta$ and $h'$,
and set $c:=c(\beta)$ of the central charge using \eqref{eq:ct}.
If $\beta$ and $h'$ are indeterminates or generic,
then this map is an isomorphism.

Finally we can state the fact on the bosonization of 
the singular vectors 
obtained in \cite{MY:1995} and \cite{SSAFR:2005}.
By a similar argument in the $\vir_{q,t}$ case 
(or by the result \cite{FF1,FF2}),
there exists uniquely up to scalar multiplication 
a singular vector of level $r s$ 
in the Verma module $M(c(\beta),h'_{r,s}(\beta))$
for $r,s\in\bbZ_{\ge1}$.
Let us denote this vector by $v'_{r,s}$.
We normalize it by 
\begin{equation}\label{eq:vrs':normalize}
 v'_{r,s} = \Bigl(L_{-1}^{r s} + 
 \sum_{\lambda \vdash r s, \ \lambda \neq (1^{r s})}
 \chi'_\lambda L_{-\lambda}\Bigr).1_{c(\beta),h'_{r,s}(\beta)},\quad
 \chi'_\lambda \in \bbQ(\beta^{1/2}).
\end{equation}

\begin{fct}\label{fct:sing:jack}
\begin{enumerate}
\item 
 $\iota'_{\beta,h'_{r,s}(\beta)}(v_{r,s}')$ 
 is proportional the Jack symmetric function:
\begin{align*}
\iota'_{\beta,h'_{r,s}(\beta)}(v_{r,s}') = B'_{r,s}(\beta) J_{(s^r)}(1/\beta)
\end{align*}
with some $B'_{r,s}(\beta) \in \bbQ(\beta)$.
\item
 The proportional factor $ B'_{r,s}(\beta)$ in the above equation is equal to 
\begin{align*}
 B'_{r,s}(\beta) = \prod_{i=1}^r\prod_{j=1}^s (i \beta - j).
\end{align*}
\end{enumerate}
\end{fct}

Here we recall the definition and some properties of Jack symmetric function. 
See \cite[\S VI.10]{M:1995} for the detail. 
Let $\beta$ be an indeterminate and define an inner product on 
$\Lambda_{\bbQ(\beta)}$ by 
\begin{align*}
\langle p_\lambda,p_\mu \rangle_\beta :=
\delta_{\lambda,\mu} z_\lambda \beta^{\ell(\lambda)}.
\end{align*}
Then the (monic) Jack symmetric function $P_{\lambda}(\beta)$ is determined 
uniquely by the following two conditions:
\begin{itemize}
\item 
It has an expansion via monomial symmetric function $m_\nu$ in the form
\begin{align*}
P_{\lambda}(\beta)
=m_\lambda+\sum_{\mu<\lambda}M'(P,m)_{\lambda,\mu}(\beta)m_\mu.
\end{align*}
Here $M'(P,m)(\beta)\in\bbQ(\beta)$ and 
the ordering $<$ is the dominance ordering.

\item 
The family of Jack symmetric functions is an orthogonal basis 
with respect to $\langle \cdot,\cdot \rangle_\beta$:
\begin{align*}
\langle P_\lambda(\beta),P_\mu(\beta) \rangle_\beta=0 \quad 
\text{ unless } \lambda = \mu.
\end{align*}
\end{itemize}
The integral Jack symmetric function $J_\lambda(\beta)$ 
is defined to be 
\begin{align}\label{eq:jack:JP}
J_\lambda(\beta):=
P_\lambda(\beta) \cdot 
\prod_{\square\in\lambda}(\beta a_\lambda(\square)+\ell_\lambda(\square)+1)
\end{align}

\subsection{Proof of the formula of the normalization coefficient}
\label{subsect:sing:norm}

Recall the degeneration procedure from $\vir_{q,t}$ to $\vir$ 
given in \eqref{eq:expand}.
Set $q=e^{\hbar \ve_1}$, $t=e^{\hbar \ve_2}$ and 
take the limit $\hbar \to 0$.
Then the generators $T_n$ of $\vir_{q,t}$ has the expansion
$$
 T_n  = 2 \delta_{n,0} + 
  \ve_1 \ve_2
  \Bigl( L_n  + \dfrac{1}{4}(\ve_1-\ve_2)^2 \delta_{n,0} \Bigr) 
  \hbar^2 +  O(\hbar^4),
$$
where $L_n$'s satisfy the relation of $\vir$ 
with the central charge 
$c=13-6\bigl(\beta +\beta^{-1}\bigr)$, $\beta = \ve_2/\ve_1$.
This procedure induces a $\bbQ$-linear map
\begin{align}\label{eq:lim_M}
\Lim_M: M(h) \longrightarrow M'(c,h'),\quad
 T_{-\lambda}.1_{h} \longmapsto 
 (\ve_1\ve_2)^{\ell(\lambda)} L_{-\lambda}.1_{c,h'}
\end{align}
respecting the graded structures of Verma modules,
where we have $h' = 2+\ve_1\ve_2\hbar^2(h'+(\ve_1-\ve_2)^2/4)$.
It is compatible with the algebra action, i.e.,
$$
\Lim_M(T_n.v) = 
  \ve_1 \ve_2 
  \bigl( L_n  + \dfrac{1}{4}(\ve_1-\ve_2)^2 \delta_{n,0} 
  \bigr).\Lim_M(v)
$$
holds for any $n \in \bbZ \setminus \{0\}$ and $v \in M(h)$.

Now we have the following compatibility 
between the $\bbQ$-linear maps $\lim_M$, $\iota_h$ and $\iota'_{c',h'}$:

\begin{lem}
The following diagram of $\bbQ$-vector spaces is commutative.
\begin{align*}
\xymatrix{
 M(h) \ar[r]^{\iota_{h}} \ar[d]_{\lim_M} 
&\Lambda_{\bbF_\alpha} \ar[d]^{\lim_{\Lambda}}
\\
 M'(c,h') \ar[r]_{\iota'_{c,h'}} 
&\Lambda_{\bbQ(\ve_1,\ve_2, \alpha)} 
}
\end{align*}
Here $\lim_{\Lambda}$ is defined by 
$$
 f(q,t) \cdot v \longmapsto f_{2n}(\ve_1,\ve_2) \cdot v
$$ 
for $f(q,q^\alpha,t) \in \bbF_\alpha$ and $v \in \Lambda^{(n)}$,
where $f_{2n}(\ve_1,\ve_2,\alpha) \in \bbQ(\ve_1,\ve_2,\alpha)$ 
is the coefficient of $\hbar^{2 n}$ in the $\hbar$-expansion 
of $f(q,q^\alpha,t)$, i.e.,
$f(e^{\hbar \ve_1},e^{\hbar \ve_1  \alpha},e^{\hbar \ve_2})
 =\sum_{n}\hbar^{n}f_n(\ve_1,\ve_2,\alpha)$.
\end{lem}

We want to apply these maps to the singular vector 
$v_{r,s} \in M(h_{r,s})$.
Recall the normalization \eqref{eq:vrs:normalize} 
of the singular vector $v_{r,s}$
and that \eqref{eq:vrs':normalize} of $v'_{r,s} \in M'(h'_{r,s}(\beta))$.
Under the map $\Lim_M: M(h_{r,s}) \to M'(c(\beta),h'_{r,s}(\beta))$ 
we have 
$$
 \Lim_M(v_{r,s}) = (\ve_1\ve_2)^{2r s} v'_{r,s}.
$$
On the other hand,
the relation of the Macdonald and Jack symmetric functions is
given in \cite[Chap.VI \S8 Remarks (8.4.vi)]{M:1995} as 
$$
 J_\lambda(\beta) = 
 \lim_{t=q^\beta, q \to 1} (1-t)^{-|\lambda|}J_\lambda(q,t).
$$
Then we immediately have 
\begin{align}\label{eq:limL:J-J}
 \Lim_{\Lambda} J_{(s^r)}(q,t) = (-\ve_2)^{r s} J_{(s^r)}(\beta).
\end{align}
Therefore we have the commutative diagram
\begin{align*}
\xymatrix{
  v_{r,s} \ar@{|->}[rr]^(0.4){\iota_{h}} \ar@{|->}[d]_{\lim_M} 
&&B_{r,s}(q,t) \cdot J_{(s^r)}(q,t)      \ar@{|->}[d]^{\lim_{\Lambda}}
\\
 (\ve_1 \ve_2)^{r s} v'_{r,s} \ar@{|->}[rr]_(0.4){\iota'_{c,h'}} 
&&(\ve_1 \ve_2)^{r s} B'_{r,s}(\beta) \cdot J_{(s^r)}(\beta)
}
\end{align*}
Using 
\begin{align*}
\Lim_{\Lambda} \prod_{i=1}^r\prod_{j=1}^s (q^i-t^j)/(q^i t^j)
=\prod_{i=1}^r\prod_{j=1}^s (j\ve_1-i \ve_2)
=(-\ve_1)^{r s} \prod_{i=1}^r\prod_{j=1}^s (i \beta-j)
=(-\ve_1)^{r s} B'_{r,s}(\beta)
\end{align*} 
and \eqref{eq:limL:J-J},
we have
$$
 \Lim_{\Lambda}\dfrac{B_{r,s}(q,t)}{\prod_{i=1}^r\prod_{j=1}^s (q^i-t^j)/(q^i t^j)}
=1.
$$
Thus $B_{r,s}(q,t) = \prod_{i=1}^r\prod_{j=1}^s (q^i-t^j)/(q^i t^j) + o(\hbar^{r s})$.
Since $B_{r,s}(q,t) \in \bbF$, we have 
$B_{r,s}(q,t) = \prod_{i=1}^r\prod_{j=1}^s (q^i-t^j)/(q^i t^j)$,
which is the desired result.

\section{Norm of Whittaker vector}
\label{sec:norm}

\subsection{Recursive formula}

Let us recall the definition of the Whittaker vector for 
the Verma module of $\vir_{q,t}$.

\begin{dfn}
Let $\xi$ an indeterminate (or an element of $\bbF$ or of some ring).
The Whittaker vector $w_{\xi}$ for $M(h)$ is defined to be an element 
of the completed Verma module $\widehat{M}(h) := \prod_{n} M(h)^{(n)}$
satisfying the following conditions.
\begin{align*}
T_1 w_{\xi}=\xi w_{\xi},\quad
T_n w_{\xi}=0\ (n\ge2).
\end{align*}
We use the notation $w_\xi =w_\xi(q,t,h)$ to indicate the parameter dependence 
if it is necessary.
\end{dfn}

\begin{lem}
If $M(h)$ is irreducible and $\xi$ is an indeterminate,
then the Whittaker vector $w$ uniquely exists up to scalar multiplication.
\end{lem}
\begin{proof}
This is the consequence of 
the non-vanishing of the Kac determinant for $\vir_{q,t}$.
In fact, one can write $w_\xi$ in the form 
\begin{align}\label{eq:w:expand}
 w_\xi = \sum_{n\ge0} \xi^n w_n, \quad w_n \in M(h)^{(n)}.
\end{align}
The elements $w_n$ are independent of the parameter $\xi$ 
and satisfy the following relation.
$$
 T_1 w_n = w_{n-1}, \quad T_m w_{n}=0\ (m\ge2).
$$ 
\end{proof}

Hereafter we normalize the Whittaker vector $w_\xi$ 
by the condition 
$$
 w_ 0 = 1_h \in M(h)^{(0)}
$$
in the expansion \eqref{eq:w:expand}.

In \cite{AY:2010}, Awata and Yamada conjectured that  
the norm $\pair{w_\xi}{w_\xi}$ of $w_\xi$ with respect to the contravariant form 
is equal to the five-dimensional pure $\SU(2)$ 
Nekrasov partition function.
To express $\pair{w_\xi}{w_\xi}$, it is convenient to introduce the parameter $Q$ by 
\begin{align*}
h=Q^{1/2}+Q^{-1/2}.
\end{align*}
By the Kac determinant formula \eqref{eq:Kac},
$\pair{w_\xi}{w_\xi}$ is a rational function of $q,t$ and $h^2$.
Thus it is a rational function of $q,t$ and $Q$.

In \cite{Y:2010} we conjectured a recursive formula of $\pair{w_\xi}{w_\xi}$.
The main result of this paper is to give a proof of the formula.

\begin{thm}\label{thm:main}
Let $\xi$ be an indeterminate  and
$w_\xi \in M(h)$ be the Whittaker vector.
Then we can express $\pair{w_\xi}{w_\xi}$ as
$$
 \pair{w_\xi}{w_\xi} = \sum_{n=0}^\infty (\xi^{2}t/q)^{n} F_n(Q,q,t)
$$ 
with $F_n(Q,q,t) \in \bbQ(Q,q,t)$. 
Moreover $F_n(Q,q,t)$ satisfies the following recursive relation.
\begin{align}
\label{eq:rec}
F_n(Q,q,t)
=\delta_{n,0}+\sum_{\substack{r,s\in\bbZ, \\ 1\le r s\le n}}
 \dfrac{A_{r,s}(q,t)F_{n-r s}(q^r t^{s},q,t)}{Q-q^r t^{-s}}.
\end{align}
Here $A_{r,s}(q,t)$ is the following rational function.
\begin{align*}
A_{r,s}(q,t) := -\sign(r) q^r t^{-s} 
\prod_{\substack{-|r|\le i\le|r|-1, \\ -|s|\le j\le|s|-1, \\ (i,j)\neq (0,0)}}
\dfrac{1}{1-q^i t^{-j}}
\end{align*}
with $\sign(r)=1$ if $r>0$ and $\sign(r)=-1$ if $r<0$.
\end{thm}

\begin{rmk}
In the non-deformed Virasoro limit $q=t^\beta$, $t\to 1$, 
the recursion formula is proposed by Al.~B.~Zamolodchikov. 
A proof in this limit is given in \cite{Y:2011}.
\end{rmk}

\subsection{Reduction to the analysis of singular vectors}

By the definition of the Whittaker vector, we have 
\begin{align*}
\pair{w_\xi}{w_\xi}=\sum_{n=0}^\infty \xi^{2n} (K_n^{-1})_{(1^n),(1^n)}.
\end{align*}
Here $K_n^{-1}$ is the inverse matrix of the matrix 
consisting of the values of contravariant forms on $M(h)$: 
\begin{align}\label{eq:Kn}
 K_n := \left(\pair{T_{-\lambda}.1_{h}}{T_{-\mu}.1_{h}}\right)_{\lambda,\mu \, \vdash n}.
\end{align}
The index $(1^n)$ indicates the position of the element 
of the inverse matrix. 
Thus we need to prove the recursive formula of \eqref{eq:rec} for $(K_n^{-1})_{(1^n),(1^n)}$.
Noticing that $(K_n^{-1})_{(1^n),(1^n)} \in \bbQ(Q,q,t)$,
we use the notation $(K_n^{-1})_{(1^n),(1^n)}(Q,q,t)$ to indicate the parameter dependence.

Recall that for $h=h_{r,s}:=q^{-(s+1)/2}t^{(r+1)/2}$ 
with some $r,s \in \bbZ_{\ge1}$ 
the Verma module $M(h_{r,s})$ has the singular vector $v_{r,s}$
with the normalization \eqref{eq:vrs:normalize}.
By the definition of the Verma module there exists 
$\Phi_{r,s}(q,t) \in \vir_{q,t}^{(-)}$
such that
$$
 \Phi_{r,s}(q,t).1_{h_{r,s}} = v_{r,s}.
$$

\begin{prop}
The rational function $(K_n^{-1})_{(1^n),(1^n)}(Q,q,t)  \in \bbQ(Q,q,t)$ 
satisfies the following recursive formula. 
$$
 (K_n^{-1})_{(1^n),(1^n)}(Q,q,t) 
=\delta_{n,0}+\sum_{\substack{r,s\in\bbZ, \\ 1\le r s\le n}}
 \dfrac{R_{r,s}(q,t)^{-1} \cdot 
 (K_{n-r s}^{-1})_{(1^{n-r s}),(1^{n-r s})}(q^r t^{s},q,t)}{Q-q^r t^{-s}}.
$$
Here $R_{r,s}(q,t)$ is determined by the expansion 
$$
 \pair{\Phi_{|r|,|s|}(q,t).1_h}{\Phi_{|r|,|s|}(q,t).1_h} 
= R_{|r|,|s|}(q,t) \cdot  (h-h_{|r|,|s|})  + o(h-h_{|r|,|s|})
$$
of the norm of the element $\Phi_{|r|,|s|}(q,t).1_h \in M(h)$ 
with respect to the highest weight $h$.
\end{prop}
\begin{proof}
We follow the argument of \cite[Lemma 3.2]{ES:1995}, 
which showed that only simple poles may occur 
in matrix elements of the inverse matrix 
of the contravariant forms 
for Verma modules of finite dimensional Lie algebra.

Let $h=Q+Q^{-1}$ be a highest weight and denote by $K_n(Q) := K_n$ the matrix \eqref{eq:Kn}.
The Kac-type formula \eqref{eq:Kac} says
$\det K_n(Q) \propto \prod_{k=1}^n
 \prod_{r | k} \chi_{r,k/r}(h)^{p(n-k)}$
with 
$\chi_{r,s}(h) := h^2-h_{r,s}^2$.

Fix $k,r \in \bbZ{\ge1}$ such that $k \le n$ and $r | k$.
Take a highest weight $h_0=Q_0+Q_0^{-1} $ 
such that $\chi_{r',k'/r'}(h)=0$ if and only if 
$k'=k$ and $r'=r$.
(For $\bbF = \bbQ(q,t)$, it means $h_0=h_{r,k/r}$ and $Q_0=q^{k/2 r}t^{-r/2}$.)
Then $M(h_0)$ is reducible and contains a unique maximal submodule $N(h_0)$
generated by a singular vector $v_{r,k/r}$ 
by Proposition \ref{prop:hom=1}.

Let $z$ be a generic highest weight, i.e., 
$\chi_{r',s'}(z) \neq 0$ for any $r',s' \in \bbZ_{\ge1}$.
Let $\tau$ be an indeterminate.
Using the $\vir_{q,t}^{(0)}$-valued bilinear form $\calF$ 
(see Definition \ref{dfn:Shapovalov} (1) for the definition),
we can introduce 
a new $\bbK[\tau]$-valued form $K^\tau$ on $M(h_0)$ by 
$K^\tau(x.1_{h_0}, y.1_{h_0}) := \calF(x,y)(h_0+\tau z)$ 
for $x, y \in \vir_{q,t}$.

Set $M:=\dim M(h_0)^{(n)} = p(n)$ and 
$N:=\dim N(h_0)=p(n-r s)$.
Choose a basis $\{u_i \mid 1\le i \le M\}$ of $M(h_0)^{(n)}$ so that 
$\{u_i \mid 1\le i \le N\}$ forms a basis of $N(h_0)^{(n)}$.
Then the elements $K^\tau(u_i,u_j)$ will be divisible by $\tau$ 
if $i \le N$ or $j\le N$.

Thus the minor determinant of size $(N-1) \times (N-1)$ of 
$K_n^\tau(Q_0) := (K^\tau(u_i,u_j))_{i,j=1}^M$
is divisible by $\tau ^{N-1}$.
The Kac-type formula \eqref{eq:Kac} implies that 
$\det K_n^\tau$ is divisible exactly by $\tau^N$,
so that every matrix element of ($K_n^\tau(h))^{-1}$ has only simple pole 
at $\tau = 0$.
Thus for any $i$ and $j$ the element $(K_n(Q_0)^{-1}_{i,j})_{i,j}$ 
has at most simple poles on 
the hyperplanes $\chi_{r,k/r}(h)=0$.

Now we may take $\{u_i\}$ such that $u_M = T_{-1}^M.1_h$.
By Proposition \ref{prop:hom=1} 
we may also set 
$\{u_i \mid 1\le i \le N\}= 
 \{ T_{-\mu}.1_h \mid \mu \, \vdash n-r s\}$.
Then, using $h-h_{0} = Q^{-1}(Q-Q_0)(Q-Q_0^{-1})$, 
we find 
${K_n^\tau(Q)}^{-1}_{M,M} = 
 (Q-Q_0)^{-1} R_{r,k/s}(q,t)^{-1} \cdot {K_{n-k}^\tau(Q_0^{-1})}^{-1}_{N,N}$.
This is the desired result.

\end{proof}


Now for the proof of the main theorem,
it is enough to determine the factor $R_{r,s}(q,t)$.

\begin{prop}\label{prop:Rrs}
We have 
\begin{align}\label{eq:Rrs}
 R_{r,s}(q,t) = 
 -\sign(r) q^{-r} t^{s} (q/t)^{-r s}
\prod_{\substack{-|r|\le i\le|r|-1, \\ -|s|\le j\le|s|-1, \\ (i,j)\neq (0,0)}}
(1-q^i t^{-j})
\end{align}
\end{prop}

The proof will be given in the next subsection.

\subsection{Proof of Proposition \ref{prop:Rrs}}

In \cite{Y:2011} we proved the degenerated version of Proposition \ref{prop:Rrs}.
We will use the parametrization $c= c(\beta) = 13-6(\beta+\beta^{-1})$ 
for the central charge.
For $r,s \in \bbZ_{\ge1}$, 
let $\Phi'_{r,s}(\beta) \in U(\vir^{(-)})$ be such that 
$$
 v'_{r,s} = \Phi'_{r,s}(\beta).1_{c(\beta),h'_{r,s}(\beta)}.
$$
Define $R'_{r,s}(\beta) \in \bbQ(\beta)$ by 
$$
 \pair{\Phi'_{r,s}(\beta).1_{c(\beta),h'(\beta)}}{\Phi'_{r,s}(\beta).1_{c(\beta),h'(\beta)}}
= R'_{r,s}(\beta) \cdot (h'-h'_{r,s}(\beta)) + o(h'-h'_{r,s}(\beta)).
$$
Then the main result in \cite{Y:2011} is   
\begin{align}\label{eq:R'rs}
R'_{r,s}(\beta) = 2\prod_{\substack{(k,l) \in \bbZ^2, \\  1-r \le k \le r,  \ 
    1-s \le l \le s,  \\   (k,l)\neq(0,0),(r,s).}} 
 (k \beta^{-1/2}+ l \beta^{1/2}).
\end{align}

Recall the limit procedure $\Lim_M:M(h) \to M'(c,h')$ given in \eqref{eq:lim_M}.
Since it comes from the algebra limit $\vir_{q,t} \to U(\vir)$,
it induces another degeneration limit 
$$
 \Lim_S: \bbF[h] \to \bbQ(c)[h'],\quad
$$ 
on the Shapovalov forms on $M(h)$ and on $M(c,h')$.
The new limit $\Lim_S$ satisfies 
$$
 \Lim_S \pair{u}{v} = \pair{\Lim_M u}{\Lim_M v}'.
$$
It can be calculated by replacing 
$q=e^{\hbar \ve_1}$, $t=e^{\hbar \ve_2}$, 
taking the limit $\hbar \to 0$,
and take the $\hbar^{2n}$ part on $u,v \in M(h)^{(n)}$.
Recall also that we have 
$h = 2+\ve_1\ve_2\hbar^2(h'+(\ve_1-\ve_2)^2/4)$ 
and $c=13-6(\beta+\beta^{-1})$, $\beta =\ve_2/\ve_1$.

Now by the direct calculation using \eqref{eq:Rrs} and \eqref{eq:R'rs} 
we see that 
$$
 \Lim_S \dfrac{R_{r,s}(q,t)}{\text{RHS of \eqref{eq:Rrs}}} =
 \dfrac{R'_{r,s}(\beta)}{\text{RHS of \eqref{eq:R'rs}}} = 1.
$$
Since $R_{r,s}(q,t) \in \bbF$, 
this degeneration limit uniquely determines $R_{r,s}(q,t)$,
so that we get the final result.

\subsection{Another Proof}

Here we give a sketch of another proof of Theorem \ref{thm:main}.
We follow  the line of the argument in \cite{Y:2011}.
It consists of the following steps.
\begin{description}
\item[Step 1]
Estimate the degree of $R_{r,s}(q,t)$ as the Laurent polynomial of $q$
using the degenerate limit to the non-deformed Virasoro Lie algebra case.

\item[Step 2]
Determine the set $S$ of the zeros of $R_{r,s}(q,t)$ with respect to $q$. 
and it is divided into three sub-steps:
\begin{itemize}
\item 
 An upper bound of $\# S$ is obtained from the degree estimation 
 in Step 1.

\item 
 Show that $S$ includes a certain subset $S_0$ using bosonization. 
 $S_0$ is determined from a coefficient $B_{r,s}(q,t)$ of the Macdonald symmetric function 
 appearing in the bosonization of the singular vector $v_{r,s}$ 
 (Theorem \ref{thm:sing:norm}).  
 We also show that $S$ is invariant under the action $(q,t) \mapsto (1/t,q)$. 
 Thus $S$ contains $S_0\cup \tau S_0$, where $\tau S_0 := \{t^{-1/d} \mid t^{d} \in S_0\}$.

\item 
 Since $\#(S_0\cup \tau S_0)$ is equal to the upper bound, 
 $S$ should be equal to $S_0 \cup \tau  S_0$.
\end{itemize}

As the result of Step 1 and Step 2, 
$R_{r,s}(q,t)$ is determined up to a factor depending only on $t$.

\item[Step 3] 
Determine the $t$-factor of $R_{r,s}(q,t)$.
It can be done by the degree estimate of $R_{r,s}(q,t)$ with respect to $t$.
\end{description}

\appendix
\section{Jantzen filtration and classification of weights}

Following the argument of \cite{FF2,IK} in the case of the Virasoro Lie algebra,
we describe the classification of highest weights 
respecting the Jantzen filtration.

\subsection{Jantzen Filtration}

Let us recall the notion of Jantzen filtration \cite{J}.
Let $R$ be a commutative algebra over a field $\bbK$.
Assume that $R$ is a PID.
Denote by $Q(R)$ the quotient field of $R$.
Fix a uniformizer $\tau$ of $R$.
Let 
$$
 \nu_\tau:R \to \bbZ_{\ge0} \cup \{\infty\}
$$ 
be a $\tau$-adic valuation,
and denote the residue field by $K := R/\tau R$.
Let us introduce the functor 
$$
 \phi(M) := K \otimes_R M
$$ 
{}from the category of $R$-modules 
to the category of $K$-vector spaces.
For an element $v$ in an $R$-module $M$, 
we also use the same symbol $\phi$ in the meaning of 
$\phi(v) := 1\otimes v \in \phi(M)$.

Assume that $N$ is a finite-dimensional $Q(R)$-vector space 
having a non-degenerate symmetric bilinear form 
$(-,-)_N: N \otimes N \to Q(R)$.
Assume also that there is an $R$-sublattice $N_R$ of $N$ 
such that 
$$
 (N_R,N_R)_N \subset R,\qquad 
 Q(R) \otimes_R N_R = N.
$$
Since we have assumed $N$ is finite-dimensional,
$N_R$ is a free $R$-module of rank $r$ with $r := \dim_{Q(R)} N$.
Choose an $R$-free basis $\{e_1,e_2,\ldots,e_r\}$ of $N_R$ and set 
$$
 D_R := \det \bigl( (e_i,e_j)_N \bigr)_{1\le i,j \le r}.
$$
This determinant is unique up to multiplication by units in $R$,
so $\nu_\tau(D_R)$ is independent of the choice of the basis.

Now define the symmetric bilinear form $(-,-)$ on the $K$-vector space 
$N_K := \phi(N_R)$ by 
$$
 (\phi(v),\phi(w)) := \pi\bigr((v,w)_N\bigr),\quad
 v,w \in N_R.
$$
Here $\pi: R \to K$ is the projection to the quotient field.
This form is non-degenerate.

\begin{dfn}
For a non-negative integer $d$, we set
$$
 N_R(d) := \{ v \in N_R \mid (v,N_R)_N \subset \tau^d R\}
$$
and denote by $\iota_d: N_R(d) \hookrightarrow N_R$ the inclusion.
Further we define
$$
 N_K(d) := \phi \circ \iota_d\bigl(N_R(d)\bigr).
$$
Then we have a filtration 
$$
 N_K = N_K(0) \supset N_K(1) \supset N_K(2) \supset \cdots
$$
in the $K$-vector space $N_K$.
This filtration is called the Jantzen filtration.
\end{dfn}

\begin{lem}
\begin{enumerate}
\item 
We have
$$
 N_K(1) = \Rad_{(-,-)}. 
$$
\item
For each positive integer $d$, 
the symmetric bilinear form  $(-,-)_d: N_K(d) \otimes N_K(d) \to K$ defined by 
$$
 \bigl( \phi(v),\phi(w) \bigr)_d := \pi\bigl(\tau^{-d}(v,w)_{N} \bigr),\quad
 v,w \in N_R(d)
$$
is non-degenerate.
\item
We have
$$
\nu_\tau(D_R) = \sum_{d\ge1}\dim_K N_K(d).
$$
\end{enumerate}
\end{lem}

\begin{rmk}
The assumption of PID for $R$ is used only in the proof of (3).
\end{rmk}

We may apply this formulation to $M(h)^{(n)}$ for each $n \in \bbZ_{\ge 0}$
by the following argument.
Let $T$ be an indeterminate and set $R := \bbF[T]$,
so that $Q(R)=\bbF(T)$ and $K = \bbF$.
Set $N :=  R \otimes_K M(h)^{(n)}$, which is finite-dimensional.
We have $N_R = N_K = M(h)^{(n)}$.
The bilinear form  $(-,-)_N$ is set to be the induced one from 
$\langle -,-\rangle$ on $M(h)^{(n)}$ by linearity, 
which is non-degenerate 
since $\langle -,-\rangle$ is so by the Kac determinant formula \eqref{eq:Kac}.
Thus we have the Jantzen filtration $\{M(h)^{(n)}(d) \mid d\ge 0\}$ 
on the graded component $M(h)^{(n)}$ of the Verma module for each $n$.

\begin{dfn}
Taking the direct sum over $n$,
we have a filtration on the Verma module $M(h)$.
This  will be called the Jantzen filtration on $M(h)$ 
and denoted by $\{M(h)(d) \mid d \ge 0\}$.
\end{dfn}

\begin{lem}
For each $d \ge 0$,
$M(h)(d)$ is a $\vqt$-submodule of the Verma module $M(h)$.
In particular,
$M(h)(1)$ is a maximal submodule of $M(h)$.
\end{lem}

\subsection{Classification of highest weights}

Hereafter we specialize $q,t$ to complex numbers.
The defining field $\bbF=\bbQ(q,t)$ is replaced by $\bbC$,
and the highest weight $h$ of the Verma module $M(h)$ is also 
considered to be a complex number.
For the algebra $\vqt$ to be well defined,
we assume that $q/t$ is not a root of unity.
In this setting, 
let us study the irreducibility of $M(h)$ in detail.

We want to determine the conditions for the highest weight $h$ 
to classify the behavior of Jantzen filtration of the Verma module $M(h)$.

Looking at the Kac determinant formula \eqref{eq:Kac},
we set 
$$
 D(q,t,h) := 
  \bigl\{(r,s) \in (\bbZ_{\ge1})^2 \mid
         h^2 - h_{r,s}^2 = 0 \bigr\}.
$$

\begin{dfn}
The class of $(q,t,h)$ is defined by the following condition on $D(q,t,h)$.
\begin{description}
 \item[Class $V$]  $D(q,t,h) = \emptyset$. 
 \item[Class $I$] $\# D(q,t,h) = 1$. 
 \item[Class $R$] $\# D(q,t,h) \ge 2$.
\end{description}
\end{dfn}

\begin{rmk}
\begin{enumerate}
\item
The symbols V, I and R mean vacant, irrational and rational respectively.
This naming follows \cite[Chap 5]{IK}.
\item
If $(q,t,h)$ belongs to \cls{V}, then $M(h)$ is irreducible.
\end{enumerate}
\end{rmk}

Now we write down a (rough) classification of  $(q,t,h)$ 
according to the values of $q$ and $t$.
We choose branches of $\arg q$ and $\arg t$ in $[0,2\pi)$.

\begin{lem}\label{lem:class}
Assume $q,t \neq0$ and neither $q$ nor $t$ is a root of unity.
Then the class of $(q,t,h)$ is classified in the following manner.
\begin{enumerate}
 \item $|q| \neq 1$. 
 \begin{enumerate}
  \item $\log|t| / \log |q| \notin \bbQ$ or 
        $\arg t - \arg q \log|t| / \log |q| \notin 2\pi \bbQ$.
    \cls{V} or $I$.
  \item $\log|t| / \log |q| \in \bbQ \setminus\{0\}$ 
        and $\arg t - \arg q \log|t| / \log |q| \in 2\pi \bbQ$.
    \cls{V} or $I$ or $R$.
  \item $|t|=1$.
  \begin{enumerate}
   \item $\arg t \notin 2 \pi \bbQ$.
    \cls{V} or $I$.
   \item $\arg t \in 2 \pi \bbQ$.
    We don' consider this case here. 
  \end{enumerate}
 \end{enumerate}
 \item $|q| = 1$. 
 \begin{enumerate}
  \item $|t| \neq 1$.
  \begin{enumerate}
   \item $\arg q \notin 2 \pi \bbQ$.
    \cls{V} or $I$.
   \item $\arg q \in 2 \pi \bbQ$.
    We don' consider this case here. 
  \end{enumerate}
  \item $|t| = 1$.
  \begin{enumerate}
   \item $\arg q \notin 2 \pi \bbQ$ and $\arg t \notin 2 \pi \bbQ$.
    \cls{V} or $I$ or $R$.
   \item the other cases
    We don' consider this case here. 
  \end{enumerate}
 \end{enumerate}
\end{enumerate}
\end{lem}

\begin{proof}
Since $x+x^{-1} = y+y^{-1} \Longleftrightarrow x=y ^{\pm1}$,
we have
\begin{align}
\nonumber
 h_{r,s}^2 = h_{r',s'}^2
&\Longleftrightarrow
 t^r q^{-s} = (t^{r'} q^{-s'})^{\pm1}
\\
\nonumber
&\Longleftrightarrow
 |t|^r |q|^{-s} = (|t|^{r'} |q|^{-s'})^{\pm1} 
 \text{ \ and \ } 
 r \arg t - s \arg q  \in  \pm(r'\arg t - s'\arg q) + 2\pi\bbZ
\\
\label{eq:class:R}
&\Longleftrightarrow
 (r \mp r')\log |t| = (s \mp s')\log |q| 
 \text{ \ and \ } 
 (r \mp r')\arg t - (s \mp s') \arg q \in 2\pi \bbZ.
\end{align}

In the case (1), we have
\begin{align}
\label{eq:class:R:case1}
 \eqref{eq:class:R} \text{ for some } (r,s) \neq (r',s')
\Longleftrightarrow
  (r\mp r')\dfrac{\log |t|}{\log|q|} \in \bbZ \text{ and } 
  (r\mp r')\Bigl(\arg t-\dfrac{\log |t|}{\log|q|} \arg q \Bigr)  \in 2\pi \bbZ.
\end{align}
Therefore in the  cases (1) (a), 
the condition $h^2=h_{r,s}^2$ determines $(r,s)$ uniquely.

The other cases can be shown by similar arguments.
\end{proof}

\subsection{\cls{R}}

The most interesting is \cls{R} with $|q|\neq1$ and $|t|\neq 1$.
We want to determine the number $\# D(q,t,h)$ precisely,
and to give an explicit parametrization of the set $D(q,t,h)$.

Let us study \cls{R} in the case (1) (b) of Lemma \ref{lem:class} first.
Set 
\begin{align*}
 h = |q|^{m/2}e^{\sqrt{-1}\theta} +  |q|^{-m/2}e^{-\sqrt{-1}\theta} .
\end{align*}
Then we have
\begin{align}
\label{eq:classR:h-theta}
 h^2=h_{r,s}^2 \Longleftrightarrow
 \pm m = -s + r \dfrac{\log|t|}{\log|q|} \text{ and }
 \pm \theta \in -s\dfrac{\arg q}{2} + r \dfrac{\arg t}{2} + \pi \bbZ.
\end{align}
Since we assume that $(q,t,h)$ is in \cls{R},
there exists at least one solution of these equations.

In the present case we may put
\begin{align}
\label{eq:classR:PQ}
 \log|t|/\log|q| = Q/P,\quad Q,P\in\bbZ\setminus\{0\},\quad 
 \gcd(Q,P)=1,\quad 
 P>0.
\end{align}
According to \eqref{eq:class:R:case1} in the proof of Lemma \ref{lem:class},
we may also put
\begin{align}
\label{eq:classR:A}
 \arg t = \dfrac{Q}{P}\arg q+2\pi\dfrac{A}{P},\quad A \in \bbZ.
\end{align}

The first condition in \eqref{eq:classR:h-theta} implies that 
it is natural to divide \cls{R} with $|q|\neq1$ and $|t|\neq1$ 
into the following two subclasses.
\begin{description}
\item[Class $R^+$] $\log|t|/\log|q| > 0$.
\item[Class $R^-$] $\log|t|/\log|q| < 0$.
\end{description}

\subsubsection{\cls{R^+}}
In this subclass,
the integer $Q$ in \eqref{eq:classR:PQ} is positive. 
The first condition $s=r Q/P \mp m$ in \eqref{eq:classR:h-theta} 
has infinite solutions for $(r,s) \in (\bbZ_{\ge1})^2$.
Let us first parametrize these solutions as 
$$
 \{(r'_i,s'_i) \mid i \in \bbZ_{\ge1}\} \cup 
 \{(r''_i,s''_i) \mid i \in \bbZ_{\ge1}\},\quad
 r'_1 < r'_2 < \cdots,\ 
 r''_1 < r''_2 < \cdots.
$$ 
Here $(r',s')$'s are solutions of $s=r Q/p-m$ and 
$(r'',s'')$'s are solutions of $s=r Q/p+m$ 
They enjoy the relations
\begin{align}
\label{eq:classR:+:rs}
 r'_i = r'_1 + P(i-1),\quad s'_i = s'_1 + Q(i-1).
\end{align}
$r''_i$'s and $s''_i$'s also satisfy the same relations.
Then we have 
$$
 D(q,t,h) = \{(r'_i,s'_i) \mid i \in \bbZ_{\ge1}\}
 \cup \{(r''_i,s''_i) \mid i \in \bbZ_{\ge1}\}.
$$
To show this claim, note that 
\begin{align*}
 \Bigl( -s'_i\dfrac{\arg q}{2} + r'_i \dfrac{\arg t}{2}\Bigr)
-\Bigl( -s'_1\dfrac{\arg q}{2} + r'_1 \dfrac{\arg t}{2}\Bigr)
= (i-1)\Bigl(-Q\dfrac{\arg q}{2} +P \dfrac{\arg t}{2}\Bigr)
=(i-1)A \pi 
\end{align*}
by \eqref{eq:classR:A} and \eqref{eq:classR:+:rs}.
Thus if some $(r'_i,s'_i)$ satisfies the second condition in \eqref{eq:classR:h-theta},
then all the others also satisfy the same condition,
which implies $(r'_i,s'_i) \in D(q,t,h)$ for any $i$.
With a similar argument for $(r''_i,s''_i)$, we have the claim.

Next we re-parametrize $D(q,t,h)$ as
$$
 D(q,t,h) = \{(r_i,s_i) \mid i \in \bbZ_{\ge1}\},\quad
 r_1 s_1 \le r_2 s_2 < r_3 s_3 \le r_4 s_4 < \cdots.
$$
One can see that $r_{2i} s_{2i} = r_{2i+1} s_{2i+1}$ holds 
if and only if $r_1/P \in \bbZ$ and $s_1/Q \in \bbZ$. 
See Figure \ref{fig:classR+} for the visual explanation.
\begin{figure}[htbp]
\centering
{\unitlength 0.1in%
\begin{picture}( 42.0000, 36.4000)(  1.0000,-38.0000)%
\special{pn 8}%
\special{pa 600 2000}%
\special{pa 4200 2000}%
\special{fp}%
\special{sh 1}%
\special{pa 4200 2000}%
\special{pa 4133 1980}%
\special{pa 4147 2000}%
\special{pa 4133 2020}%
\special{pa 4200 2000}%
\special{fp}%
\special{pn 8}%
\special{pa 2400 3800}%
\special{pa 2400 200}%
\special{fp}%
\special{sh 1}%
\special{pa 2400 200}%
\special{pa 2380 267}%
\special{pa 2400 253}%
\special{pa 2420 267}%
\special{pa 2400 200}%
\special{fp}%
\special{pn 8}%
\special{pa 600 2750}%
\special{pa 4000 200}%
\special{fp}%
\special{pn 4}%
\special{ar 800 2600 24 24 0  6.28318530717959E+0000}%
\special{ar 1000 2450 24 24 0  6.28318530717959E+0000}%
\special{ar 1200 2300 24 24 0  6.28318530717959E+0000}%
\special{ar 1400 2150 24 24 0  6.28318530717959E+0000}%
\special{sh 1}%
\special{ar 2600 1250 24 24 0  6.28318530717959E+0000}%
\special{sh 1}%
\special{ar 2800 1100 24 24 0  6.28318530717959E+0000}%
\special{sh 1}%
\special{ar 3000 950 24 24 0  6.28318530717959E+0000}%
\special{sh 1}%
\special{ar 3200 800 24 24 0  6.28318530717959E+0000}%
\special{sh 1}%
\special{ar 3400 650 24 24 0  6.28318530717959E+0000}%
\special{sh 1}%
\special{ar 3600 500 24 24 0  6.28318530717959E+0000}%
\special{sh 1}%
\special{ar 3800 350 24 24 0  6.28318530717959E+0000}%
\special{pn 8}%
\special{pa 4200 1250}%
\special{pa 800 3800}%
\special{dt 0.045}%
\special{pn 4}%
\special{sh 1}%
\special{ar 4000 1400 24 24 0  6.28318530717959E+0000}%
\special{sh 1}%
\special{ar 3800 1550 24 24 0  6.28318530717959E+0000}%
\special{sh 1}%
\special{ar 3600 1700 24 24 0  6.28318530717959E+0000}%
\special{sh 1}%
\special{ar 3400 1850 24 24 0  6.28318530717959E+0000}%
\put(41.0,-22.0){\makebox(0,0)[lb]{$r$}}%
\put(22.5,-4.0){\makebox(0,0)[lb]{$s$}}%
\put(10.5,-17.0){\makebox(0,0)[lb]{$s=r Q/P +m$}}%
\put(26.5,-25.5){\makebox(0,0)[lb]{$s=r Q/P -m$}}%
\put(27.0,-14.0){\makebox(0,0)[lb]{$(r'_1,s'_1)$}}%
\put(29.0,-12.0){\makebox(0,0)[lb]{$(r'_2,s'_2)$}}%
\put(25.0,-9.5){\makebox(0,0)[lb]{$(r'_3,s'_3)$}}%
\put(31.0,-5.5){\makebox(0,0)[lb]{$(r'_6,s'_6)$}}%
\put(33.5,-3.5){\makebox(0,0)[lb]{$(r'_7,s'_7)$}}%
\put(28.5,-19.0){\makebox(0,0)[lb]{$(r''_1,s''_1)$}}%
\put(31.0,-17.0){\makebox(0,0)[lb]{$(r''_2,s''_2)$}}%
\put(35.0,-14.0){\makebox(0,0)[lb]{$(r''_4,s''_4)$}}%
\put(15.0,-22.5){\makebox(0,0)[lb]{$(-r''_1,-s''_1)$}}%
\put(12.5,-24.5){\makebox(0,0)[lb]{$(-r''_2,-s''_2)$}}%
\put(8.50,-27.5){\makebox(0,0)[lb]{$(-r''_4,-s''_4)$}}%
\end{picture}}%
\caption{Parametrizing $D(q,t,h)$ for \cls{R^{+}}.}
\label{fig:classR+}
\end{figure}
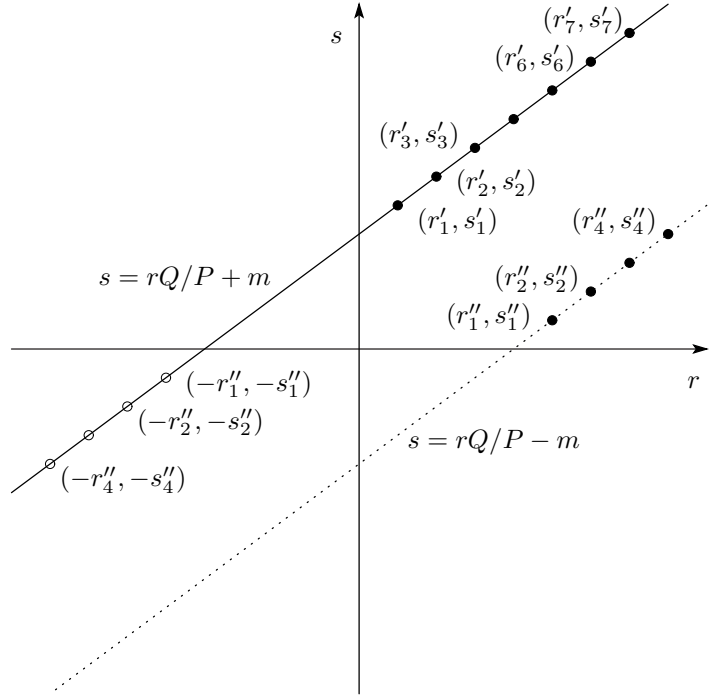

Now an easy observation (see \cite[Lemma 5.8]{IK}) for a similar argument)
gives a finer classification of highest weights in \cls{R^+}.

\begin{dfn}\label{dfn:K^+}
For a pair $(P,Q)$ of positive integers with $\gcd(P,Q)=1$, we set
\begin{align*}
 K_{P,Q}^{+}:=\{(r,s) \in \bbZ_{\ge 0}^2 \mid r < P,\ s \le Q,\ r Q+s P \le P Q\}.
\end{align*}
For such a pair $(r,s)$ and integers $r,s,i$, we also set 
\begin{align*}
 h_{P,Q;r,s,i} := t^{(r+i Q)/2}q^{-(s+i P)/2} +  t^{-(r+i Q)/2}q^{(s+i P)/2}. 
\end{align*}
\end{dfn}

\begin{lem}
Assume $(q,t,h)$ belongs to \cls{R^+}.
Then there exist a unique pair $(P,Q)$ of coprime positive integers, 
a unique pair $(r,s)$ in $K_{P,Q}^{+}$ and 
a (not necessarily unique) integer $i \in \bbZ$ such that 
$$
 \log|t|/\log|q| = Q/P,\qquad h = h_{P,Q;r,s,i}.
$$
\end{lem}

Finally we study the degeneracy of 
the highest weights $\{h_{P,Q;r,s,i} \mid i \in \bbZ \}$.

\begin{dfn}\label{dfn:R^+:cases}
Divide the set $K_{P,Q}^{+}$ as follows.
\begin{description}
\item[Case $1^+$] $0<r<P$ and $0<s<Q$.
\item[Case $2^+$] $r=0$ and $0<s<Q$.
\item[Case $3^+$] $0<r<P$ and $s=0$.
\item[Case $4^+$] $(r,s)=(0,0)$ or $(0,Q)$.
\end{description}
\end{dfn}

\begin{lem}
Write $h_i := h_{P,Q;r,s,i}$ for simplicity.
Then the degeneration of highest weights $\{h_{i} \mid i \in \bbZ \}$ 
is described as follows.
\begin{description}
\item[Case $1^+$] no degeneration.
\item[Case $2^+$] $h_{-i-1}=h_i$ for $i \ge 0$.
\item[Case $3^+$] $h_{2i}=h_{2i-1}$ for $i \in \bbZ$. 
\item[Case $4^+$] 
 $h_{2i}=h_{2i-1}=h_{-2i}=h_{-2i-1}$ for $i \ge 0$ in the case $(r,s)=(0,0)$.
 $h_{2i}=h_{2i-1}=h_{-2i-1}=-h_{-2i-2}$ for $i \ge 0$ in the case $(r,s)=(0,Q)$.
\end{description}
In particular,
the following list exhausts the highest weight $h$ such that $(q,t,h_{i})$ 
belongs to \cls{R^+}.
\begin{description}
\item[Case $1^+$] $h_i$ $(i \in \bbZ)$.
\item[Case $2^+$] $h_i$ $(i \in \bbZ_{\ge0})$.
\item[Case $3^+$] $h_{(-1)^i i}$  $(i \in \bbZ_{\ge0})$. 
\item[Case $4^+$] $h_{2i}$  $(i \in \bbZ_{\ge0})$. 
\end{description}
\end{lem}

\subsubsection{\cls{R^{-}}}
In this subclass,
the integer $Q$ in \eqref{eq:classR:PQ} is negative.
Let us replace $Q$ by $-Q$ (with $Q>0$) hereafter.
and the first condition $s=-r Q/P \mp m$ in \eqref{eq:classR:h-theta} 
has at most finite solutions for $(r,s) \in (\bbZ_{\ge1})^2$.
Let us parametrize these solutions as 
$$
 \{(r_i,s_i) \mid i =1,2,\ldots,\ell \},\qquad
 r_1 s_1 \ge r_2 s_2 \ge r_3 s_3 \ge \cdots \ge r_\ell s_\ell.
$$ 
If $r_i s_i = r_{i+1}s_{i+1}$, then we order them so that $r_i < r_{i+1}$.
See Figure \ref{fig:classR-} for a visual explanation.

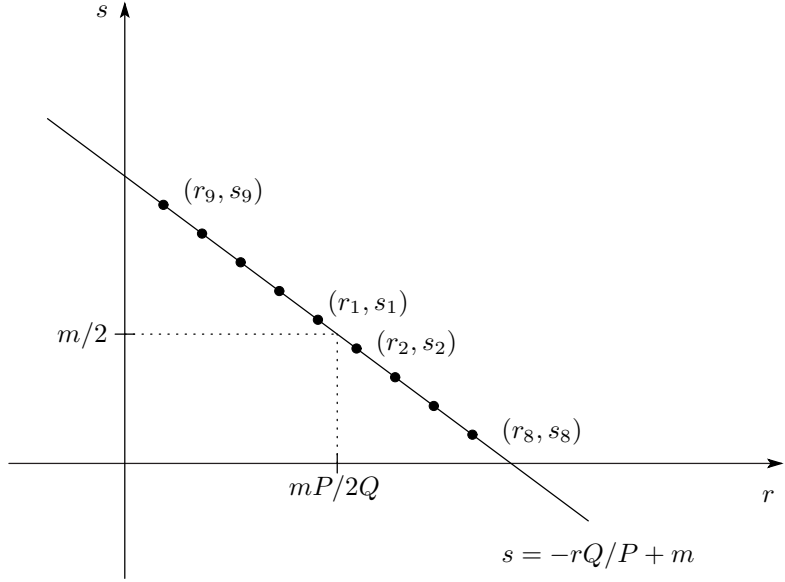
\begin{figure}[htbp]
\centering
{\unitlength 0.1in%
\begin{picture}( 38.0000, 31.0000)(  4.0000,-36.0000)%
\special{pn 8}%
\special{pa  200 3000}%
\special{pa 4200 3000}%
\special{fp}%
\special{sh 1}%
\special{pa 4200 3000}%
\special{pa 4133 2980}%
\special{pa 4147 3000}%
\special{pa 4133 3020}%
\special{pa 4200 3000}%
\special{fp}%
\special{pn 8}%
\special{pa 800 3600}%
\special{pa 800  600}%
\special{fp}%
\special{sh 1}%
\special{pa 800 600}%
\special{pa 780 667}%
\special{pa 800 653}%
\special{pa 820 667}%
\special{pa 800 600}%
\special{fp}%
\special{pn 8}%
\special{pa 1900 3000}%
\special{pa 1900 2325}%
\special{dt 0.045}%
\special{fp}%
\special{pa  800 2325}%
\special{pa 1900 2325}%
\special{dt 0.045}%
\special{fp}%
\special{pn 8}%
\special{pa 1900 2950}%
\special{pa 1900 3050}%
\special{fp}%
\special{pa 750 2325}%
\special{pa 850 2325}%
\special{fp}%
\special{pn 4}%
\special{sh 1}%
\special{ar 1000 1650 24 24 0  6.28318530717959E+0000}%
\special{sh 1}%
\special{ar 1200 1800 24 24 0  6.28318530717959E+0000}%
\special{sh 1}%
\special{ar 1400 1950 24 24 0  6.28318530717959E+0000}%
\special{sh 1}%
\special{ar 1600 2100 24 24 0  6.28318530717959E+0000}%
\special{sh 1}%
\special{ar 1800 2250 24 24 0  6.28318530717959E+0000}%
\special{sh 1}%
\special{ar 2000 2400 24 24 0  6.28318530717959E+0000}%
\special{sh 1}%
\special{ar 2200 2550 24 24 0  6.28318530717959E+0000}%
\special{sh 1}%
\special{ar 2400 2700 24 24 0  6.28318530717959E+0000}%
\special{sh 1}%
\special{ar 2600 2850 24 24 0  6.28318530717959E+0000}%
\put(41.0,-32.00){\makebox(0,0)[lb]{$r$}}%
\put(6.50,-6.700){\makebox(0,0)[lb]{$s$}}%
\put(27.50,-35.50){\makebox(0,0)[lb]{$s=-r Q/P +m$}}%
\put(16.50,-32.00){\makebox(0,0)[lb]{$m P/2 Q$}}%
\put( 4.50,-24.00){\makebox(0,0)[lb]{$m /2$}}%
\put(11.0,-16.5){\makebox(0,0)[lb]{$(r_9,s_9)$}}%
\put(18.5,-22.3){\makebox(0,0)[lb]{$(r_1,s_1)$}}%
\put(21.0,-24.4){\makebox(0,0)[lb]{$(r_2,s_2)$}}%
\put(27.5,-29.0){\makebox(0,0)[lb]{$(r_8,s_8)$}}%
\special{pn 8}%
\special{pa 3200 3300}%
\special{pa 400 1200}%
\special{fp}%
\end{picture}}%
\caption{Parameterizing $D(q,t,h)$ for \cls{R^{-}}.}
\label{fig:classR-}
\end{figure}

By the same argument as in the subclass \cls{R^{+}},
we have 
$$
 D(q,t,h) = \{(r_i,s_i) \mid i=1,2,\ldots,\ell\}.
$$

\begin{dfn}\label{dfn:K^-}
For a pair $(P,Q)$ of positive integers with $\gcd(P,Q)=1$, we set
\begin{align*}
 K_{P,Q}^{-}:=\{(r,s) \in \bbZ^2 \mid 0 \le r < P,\ 0\le s \le Q,\ r Q-s P \le P Q\}.
\end{align*}
For such a pair $(r,s)$ and integers $r,s,i$, we also set 
$$
 h_{P,Q;r,s,i} := t^{(r+i Q)/2}q^{-(s-i P)/2} +  t^{-(r+i Q)/2}q^{(s-i P)/2}. 
$$
\end{dfn}

\begin{lem}
Assume $(q,t,h)$ belongs to \cls{R^-}.
Then there exist a unique pair $(P,Q)$ of coprime positive integers, 
a unique pair $(r,s)$ in $K_{P,Q}^{-}$ and 
a (not necessarily unique) integer $i \in \bbZ$ such that 
$$
 \log|t|/\log|q| = -Q/P,\qquad h = h_{P,Q;r,s,i}.
$$
\end{lem}

\begin{dfn}\label{dfn:R^-:cases}
Divide the set $K_{P,Q}^{-}$ as follows.
\begin{description}
\item[Case $1^-$] $0<r<P$ and $0<-s<Q$.
\item[Case $2^-$] $r=0$ and $0<-s<Q$.
\item[Case $3^-$] $0<r<P$ and $s=0$.
\item[Case $4^-$] $(r,s)=(0,0)$ or $(0,-Q)$.
\end{description}
\end{dfn}

\begin{lem}
Write $h_i := h_{P,Q;r,s,i}$ for simplicity.
Then the degeneration of highest weights $\{h_{i} \mid i \in \bbZ \}$ 
is described as follows.
\begin{description}
\item[Case $1^-$] no degeneration.
\item[Case $2^-$] $h_{-i-1}=h_i$ for $i \ge 0$.
\item[Case $3^-$] $h_{2i}=h_{2i-1}$ for $i \in \bbZ$. 
\item[Case $4^-$] 
 $h_{2i}=h_{2i-1}=h_{-2i}=h_{-2i-1}$ for $i \ge 0$ in the case $(r,s)=(0,0)$.
 $h_{2i+1}=h_{2i}=h_{-2i-1}=-h_{-2i-2}$ for $i \ge 0$ in the case $(r,s)=(0,-Q)$.
\end{description}
In particular,
the following list exhausts the highest weight $h$ such that $(q,t,h_{i})$ 
belongs to \cls{R^-}.
\begin{description}
\item[Case $1^-$] $h_i$ $(i \in \bbZ \setminus \{0\})$.
\item[Case $2^-$] $h_i$ $(i \in \bbZ_{>0})$.
\item[Case $3^-$] $h_{(-1)^{i-1} i}$  $(i \in \bbZ_{>0})$. 
\item[Case $4^-$] $h_{2i}$  $(i \in \bbZ_{>0})$. 
\end{description}

\end{lem}

\begin{rmk}
$(q,t,h_0)$ belongs to \cls{V}.
\end{rmk}

\subsection{Verma module at root of unity}

This subsection is a recollection of the result of \cite{BP},
which investigated  the structure of Verma  module 
in the case $|q|=1$ or $|t|=1$.

Assume $q$ is a primitive $N$-th root of unity.
By \cite[Theorem 4.8]{BP}, the series 
\begin{align}
\label{eq:BP4.25}
\Psi(z) := \lim_{z_i \to z q^{n-i}}
           \prod_{i<j}f(z_j/z_i) T(z_1)T(z_2) \cdots T(z_N).1_{h}
\end{align}
is a well-defined generating series of singular vectors in $M(h)$.
Denote by $M(h)'$ the $\vqt$-module obtained from $M(h)$ 
by dividing out the submodule generated by the singular vectors \eqref{eq:BP4.25}.
Then by \cite[Theorem 4.11]{BP},
the Kac determinant of the module $M(h)'$ is given by
$$
 \Det_n = C_n 
  \prod_{\substack{1\le s\le N-1, \\ r \ge 1,\ r s \le n}}
   (h^2-h_{r,s}^2)^{q_N(r,s;n-r s)}
  \prod_{\substack{r,s \ge 1,\  r s \le n, \\ r \not\equiv 0 \  \text{mod}\, N}}
   \Bigl(\dfrac{1-t^{-r}}{1+p^r}\Bigr)^{p_N(n-r s)}.
$$
Here the functions $p_N$ and $q_N$ are determined by the following generating functions.
\begin{align*}
&\sum_{n \ge 0}p_N(n) x^n = \prod_{n\ge1}\dfrac{1-x^{n N}}{1-x^n},\\
&\sum_{n \ge 0}q_N(r,s;n) x^n = 
 \bigl(1+x^r+\cdots+x^{r(N-1-s)}\bigr) 
 \prod_{n\ge1,\, n \neq r} \bigl(1+x^n+\cdots+x^{n(N-1)}\bigr).
\end{align*}

\begin{fct}[{\cite[Theorem 4.12]{BP}}]
Assume $q = \sqrt[N]{1}$ and $t$ is generic.
Then $M'_h$ is irreducible if  $h^2 \neq h_{r,s}^2$ for 
any $r\ge 1$ and $1 \le s \le N-1$.
\end{fct}

\subsection{Embedding diagrams of Verma modules}

Let us illustrate the embedding structure of Verma modules.
We denote a non-trivial homomorphism $M(h) \to M(h')$ by 
\begin{align*}
\xymatrix{  [h] \\  \ar[u]  [h'] }
\end{align*}

\subsubsection{}
For $(q,t,h)$ belonging to \cls{V}, $M(h)$ is irreducible 
so that we have nothing to do.

\subsubsection{}
For $(q,t,h)$ belonging to \cls{I}, 
the Kac determinant must vanish.
If $|q|,|t| \neq 1$, then $h=h_{r,s}$ with some $r,s\in\bbZ_{>0}$.
By the definition of \cls{I}, we have $D(q,t,h) = \{(r,s)\}$.
Then recalling Proposition \ref{prop:hom=1},
we have the following claim. 

\begin{prop}
Suppose $(q,t,h_{r,s})$ with $|q|,|t|\neq 1$ belongs to Class I.
Then we have the following embedding diagram.
\begin{align*}
\xymatrix{  [h] \\  \ar[u]  [h+r s] }
\end{align*}
Moreover the submodule $M(h+r s)$ is irreducible.
\end{prop}

\subsubsection{}
For $(q,t,h)$ belonging to \cls{R^{+}},
we have $h=h_{P,Q;r,s,i}$ for some pair $(P,Q)$ of coprime positive integers,
$(r,s) \in K^+_{P,Q}$ and $i\in\bbZ$ (see Definition \ref{dfn:K^+}).

\begin{prop}
Assume $(q,t,h)$ belong ins to \cls{R^+}.
Then for each \textbf{Case} (see Definition \ref{dfn:R^+:cases}), 
$M(h)$ has the following commutative embedding diagram.
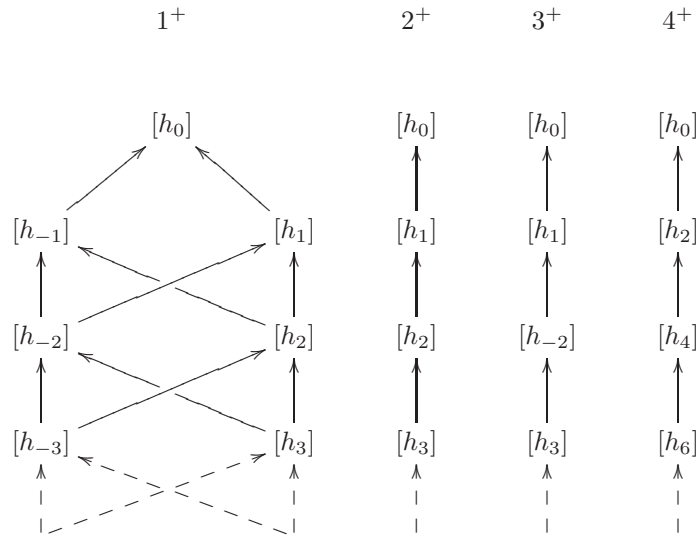
\begin{align*}
\xymatrix{
 & 1^+   &          & 2^+   & 3^+     & 4^+ 
\\
 & [h_0] &          & [h_0] & [h_0]   & [h_0] 
\\
   [h_{-1}] \ar[ur] &       & 
   [h_{1}]  \ar[ul] & [h_1] \ar[u]    & [h_1] \ar[u] & [h_2] \ar[u]
\\
   [h_{-2}] \ar[urr]  \ar[u]& & [h_{2}] \ar[ull]|!{[ll];[u]}\hole \ar[u]
 & [h_2]    \ar[u]  & [h_{-2}] \ar[u] & [h_4] \ar[u]
\\
   [h_{-3}] \ar[urr]  \ar[u]& & [h_{3}] \ar[ull]|!{[ll];[u]}\hole \ar[u]
 & [h_3]    \ar[u]  & [h_3] \ar[u]    & [h_6] \ar[u]
\\
 \ar@{-->}[urr] \ar@{-->}[u]& & \ar@{-->}[ull]|!{[ll];[u]}\hole \ar@{-->}[u]
 & \ar@{-->}[u]     & \ar@{-->}[u]    & \ar@{-->}[u]
}
\end{align*}
\end{prop}

\begin{proof}
The proof is quite similar to the Virasoro Lie algebra case, so we omit the detail.
See \cite[\S5.3.2]{IK} for the case of the Virasoro Lie algebra.
The key tool is  Proposition \ref{prop:hom=1},
by which we know that there is an embedding $\iota_{i,j}: M(h_i) \hookrightarrow M(h_j)$ 
for $i,j$ with $|i|=|j|-1$.
For the commutativity of the diagram in the \textbf{Case} $1^+$,
we note $\iota_{0,-1}\circ\iota_{-1.2}(1_{h_2}) \propto  \iota_{0,1}\circ\iota_{1.2}(1_{h_2})$
by the same Proposition.
Multiplying a scalar factor to $\iota_{-1,2}$, 
we have the commutativity $\iota_{0,-1}\circ\iota_{-1.2} =  \iota_{0,1}\circ\iota_{1.2}$.
Commutativity of the other parts of the diagram can be shown similarly.
\end{proof}

\subsubsection{}
For $(q,t,h)$ belonging to \cls{R^{-}},
we have $h=h_{P,Q;r,s,i}$ for some pair $(P,Q)$ of coprime positive integers,
$(r,s) \in K^{-}_{P,Q}$ and $i\in\bbZ$ (see Definition \ref{dfn:K^-}).

\begin{prop}
Assume $(q,t,h)$ belong ins to \cls{R^-}.
Then for each \textbf{Case} (see Definition \ref{dfn:R^-:cases}), 
$M(h)$ has the following commutative embedding diagram.
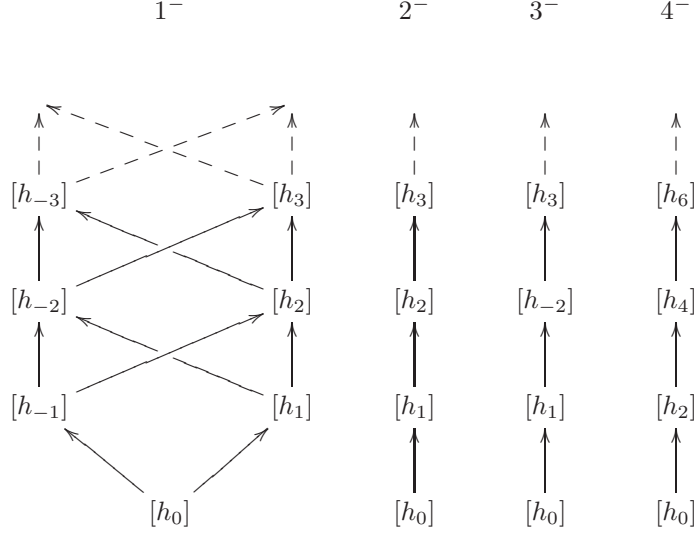
\begin{align*}
\xymatrix{
 & 1^-   &          & 2^-   & 3^-        & 4^- 
\\
&&&&&
\\
   [h_{-3}] \ar@{-->}[urr]  \ar@{-->}[u] & 
 & [h_{3}]  \ar@{-->}[ull]|!{[ll];[u]}\hole   \ar@{-->}[u]
 & [h_3]    \ar@{-->}[u]    & [h_3] \ar@{-->}[u]  & [h_6] \ar@{-->}[u]
\\
   [h_{-2}] \ar[urr]  \ar[u]& & [h_{2}] \ar[ull]|!{[ll];[u]}\hole \ar[u]
 & [h_2]    \ar[u]  & [h_{-2}] \ar[u]    & [h_4] \ar[u]
\\
   [h_{-1}] \ar[urr] \ar[u] & & [h_{1}]  \ar[ull]|!{[ll];[u]}\hole \ar[u]
 & [h_1]    \ar[u]  & [h_1] \ar[u]       & [h_2] \ar[u]
\\
 & [h_0]    \ar[ur] \ar[ul] & 
 & [h_0]    \ar[u] & [h_0] \ar[u]        & [h_0] \ar[u]
}
\end{align*}
\end{prop}

\section{Some explicit formula of the bosonized singular vector}

In this appendix we give another way to calculate the normalization factor formula
(Theorem \ref{thm:sing:norm}) in the special case $q=1$.
In this case $\vir_{1,t}$ becomes a commutative algebra 
and gets extreamly easy to treat.

In the case $q=1$ the bosonization \eqref{eq:qvir:bosonization} becomes 
\begin{align*}
T(z) =
 \exp\Bigl(\sum_{n\ge1}
            \dfrac{1-t^{-n}}{1+t^{-n}} \dfrac{p_{n}}{n} (t^{1/2}z)^{n}
      \Bigr) t^{-1/2} k
+ \exp\Bigl(- \sum_{n\ge1}
            \dfrac{1-t^{-n}}{1+t^{-n}} \dfrac{p_{n}}{n} (t^{-1/2}z)^{n}
      \Bigr) t^{1/2} k^{-1}.
\end{align*}
Here we replaced the Heisenberg generator $a_{-n}$ 
by the power sum symmetric function $p_n$,
so that the expression above is considered as 
operators acting on the space $\Lambda_{\bbQ(t)}$.
The variable $k$ is the replacement of $q^{a_0}$,
which will become $t^{-(r+1)/2}$ on the Verma module $M(h_{r,s})$.

Let $\lambda = (\lambda_1,\lambda_2,\ldots)$ be a partition.
We want to calculate 
$T_{-\lambda} = T_{-\lambda_1} T_{-\lambda_2} \cdots$ explicitly.
For $n \in \bbZ_{\le0}$ define $T^{\pm}_n$ by 
$$
 \exp\Bigl(\pm \sum_{n\ge1}
            \dfrac{1-t^{-n}}{1+t^{-n}} \dfrac{p_{n}}{n} (t^{\pm 1/2}z)^{n}
      \Bigr) t^{\mp 1/2} = \sum_{n\ge 0} T^{\pm}_{-n} z^{n},
$$
Since 
$\exp(\sum_{n\in\bbZ_{\ge1}}z^n p_n/n) 
=\sum_{n\in\bbZ_{\ge0}}z^{\lambda} p_\lambda /z_\lambda$ 
(see \eqref{eq:z_lambda} for the definition of $z_\lambda$),
we have 
$$
 T^\pm_{-n} = t^{\pm (n-1)/2} 
  \sum_{\lambda\,\vdash n} \dfrac{p_\lambda}{z^{\pm}_\lambda(t)},
 \quad
 z^{\pm}_\lambda(t) := \prod_{i\ge1} m_i(\lambda)! 
 \left(\pm i \cdot \dfrac{1+t^{-i}}{1-t^{-i}}\right)^{m_i(\lambda)}
$$
For a finite subset $J \in \bbZ$, 
define $T^{\pm}_J := \sum_{j \in J} T^{\pm}_{-j}$.
Then a direct calculation gives 
$$
 T_{-\lambda} = 
 \sum_{d=0}^{\ell(\lambda)} k^{\ell(\lambda)-2 d}
 \sum_{I \subset [1,\ell(\lambda)],\, \# I = d}
 T^+_{I^c} T^-_{I},
$$
where the summation is taken over the all subset $I$ of 
$[1,\ell(\lambda)] := \{1,2,\ldots,\ell(\lambda)\}$.  
We also set $I^c := [1,\ell(\lambda)] \setminus I$.
Noticing that $T^+_{-n}$ and $T^-_{-n}$ only differ 
at the coefficients $z^{\pm}_\mu (t)$, 
we can collect $p_\mu$'s in the above summation.
The result is 
\begin{align}\label{eq:Ttop}
 T_{-\lambda} = 
 t^{(|\lambda|-\ell(\lambda))/2} 
 \sum_{\overline{\nu} \, \vdash \lambda} \dfrac{p_{|\overline{\nu}|}}{z^+_{|\overline{\nu}|}(t)} 
 \sum_{I \subset [1,\ell(\lambda)]} k^{\ell(\lambda)-2 \# I}
 \prod_{i \in I}  (-1)^{\ell(\overline{\nu}^i)} t^{-(\lambda_i-1)}.
\end{align}
Here the running index $\overline{\nu}$ consists of $\ell(\lambda)$ partitions 
$\overline{\nu}=(\overline{\nu}^1,\overline{\nu}^2,\ldots,\overline{\nu}^{\ell(\lambda)})$ 
with $\overline{\nu}^i \vdash \lambda_i$ for each $i$.
The symbol $|\overline{\nu}|$ means the union of the partitions 
$\overline{\nu}^1 \cup \overline{\nu}^2 \cup \cdots \cup \overline{\nu}^{\ell(\lambda)}$,
which itself is a partition of $|\lambda|$. 

Now we want to expand the Macdonald symmetric function 
$J_{(s^r)}(q=1,t)$ in terms of 
$\{T_{-\lambda}.1_{h_{r,s}} \mid \lambda \, \vdash r s\}$.
By \cite[Chap.VI, \S2 (2.14), \S8 (8.6)]{M:1995} 
we have 
\begin{align}\label{eq:Je:t=1}
 J_\lambda(1,t) = e_{\lambda'} \cdot \prod_{s \in \lambda}(1-t^{\ell_\lambda(s)+1}).
\end{align}
Here $e_n$ is the $n$-th elementary symmetric function, 
and for a partition $\mu$ we set $e_{\mu} := e_{\mu_1} e_{\mu_2} \cdots$.
Thus to calculate the normalization factor,
it is enough to determined the matrix element $M(e,T)_{(s^r),(1^{r s})}$
of the transition matrix $M(e,T)$ whose elements appear in the expansion 
$$
 e_\lambda =  \sum_{\mu\,\vdash |\lambda|} 
 M(e,T)_{\lambda,\mu} \cdot T_{-\mu}.1_{h_{r,s}}. 
$$

The transition matrix $M(e,T)$ has a factorization 
$M(e,T) = M(e,p)M(T,p)^{-1}$,
where $M(e,p)$ is the matrix appearing in 
$$
 e_\lambda =  \sum_{\mu \,\vdash |\lambda|} 
 M(e,p)_{\lambda,\mu} \cdot p_{\mu},
$$ 
and $M(T,p)$ is the one in 
$$
 T_{-\lambda}.1_{h_{r,s}} =  \sum_{\mu\,\vdash |\lambda|} 
 M(T,p)_{\lambda,\mu} \cdot p_\mu. 
$$
By \cite[Chap.I \S6]{M:1995} we know 
\begin{align}\label{eq:Mep}
M(e,p) = L' Z^{-1} \ve,
\end{align}
where $L'$ is the transposed matrix of $(L_{\lambda,\mu})$ with 
$$
 L_{\lambda,\mu} = 
 \#\bigl\{f:[1,\ell(\lambda)] \to \bbZ \mid 
   \forall \, i \ \mu_i = \textstyle{\sum_{i=f(j)}} \lambda_j  \bigr\},
$$
and $Z=\diag(z_\lambda)$ (resp.\ $\ve = \diag(\ve_\lambda)$) 
is the diagonal matrix whose elements are given by 
$z_\lambda$ (see \eqref{eq:z_lambda} for the definition) 
(resp.\ $\ve_\lambda := (-1)^{|\lambda|-\ell(\lambda)}$).
The matrix element $M(T,p)_{\lambda,\mu}$ can be read from \eqref{eq:Ttop}.
The result is 
\begin{align}\label{eq:MTp}
 M(T,p)_{\lambda,\mu} = 
 t^{(|\lambda|+\ell(\lambda) r)/2}
 \sum_{\overline{\mu} \, \vdash \lambda} \prod_{i=1}^{\ell(\lambda)}
 \dfrac{1+t^{-r}\prod_{j=1}^{\ell(\overline{\mu}^i)}(-t^{-\overline{\mu}^i_j})}
 { z^+_{\overline{\mu}^i}(t) }.
\end{align}

Finally, by a direct calculation using \eqref{eq:Mep} and \eqref{eq:MTp},
we get 
\begin{align*}
M(e,T)_{(s^r),(1^{r s})}
=\prod_{i=1}^{r}\left(\dfrac{t^i}{(1-t^i)^2}\right)^s.
\end{align*}
Using \eqref{eq:Je:t=1} we have
$$
 v_{r,s} = J_{(s^r)}(1,t) \cdot \prod_{i=1}^{r}\left(\dfrac{1-t^i}{t^i}\right)^s,
$$ 
which is the desired result.


\end{document}